\documentclass{article}
\usepackage{amsmath,amssymb,enumerate,bbm,color,theorem}

\newcommand{\assign}{:=}
\newcommand{\longhookrightarrow}{{\lhook\joinrel\relbar\joinrel\rightarrow}}
\newcommand{\nin}{\not\in}
\newcommand{\tmmathbf}[1]{\ensuremath{\boldsymbol{#1}}}
\newcommand{\tmop}[1]{\ensuremath{\operatorname{#1}}}
\newcommand{\tmstrong}[1]{\textbf{#1}}
\newcommand{\tmtextit}[1]{{\itshape{#1}}}
\newenvironment{enumeratealpha}{\begin{enumerate}[a{\textup{)}}] }{\end{enumerate}}
\newenvironment{enumeratenumeric}{\begin{enumerate}[1.] }{\end{enumerate}}
\newenvironment{enumerateroman}{\begin{enumerate}[i.] }{\end{enumerate}}
\newenvironment{itemizedot}{\begin{itemize} }{\end{itemize}}
\newenvironment{proof}{\noindent\textbf{Proof\ }}{\hspace*{\fill}$\Box$\medskip}
\definecolor{grey}{rgb}{0.75,0.75,0.75}
\definecolor{orange}{rgb}{1.0,0.5,0.5}
\definecolor{brown}{rgb}{0.5,0.25,0.0}
\definecolor{pink}{rgb}{1.0,0.5,0.5}
\newtheorem{corollary}{Corollary}
\newtheorem{definition}{Definition}
\newtheorem{lemma}{Lemma}
\newtheorem{proposition}{Proposition}
{\theorembodyfont{\rmfamily}\newtheorem{remark}{Remark}}
\newtheorem{theorem}{Theorem}
\providecommand{\xRightarrow}[2][]{\mathop{\Longrightarrow}\limits_{#1}^{#2}}

\begin{document}

\title{Carlson's $<_1$-relation on the class of addtive principal
ordinals}\author{Parm\'enides Garc\'{\i}a Cornejo}\maketitle

\begin{abstract}
  This is the first in a series of at least 4 articles. We will study Carlson's
  $<_1$-relation in the whole class of ordinals and later we will link it with ordinals
  $\alpha \leqslant \left| \Pi^1_1 - \tmop{CA}_0 \right|$.
  
  The main motivation to study $<_1$ are the works of T. Carlson and G.
  Wilken. The first version $\prec_1$ of $<_1$ was used by Carlson as a tool
  to show Reinhardt's conjecture: The Strong Mechanistic Thesis is consistent
  with Epistemic Arithmetic (see {\cite{Carlson3}}); moreover, Carlson showed
  a characterization of $\varepsilon_0$ in terms of $\prec_1$ (see
  {\cite{Carlson1}}) and indeed, set up a different approach to ordinal notation
  systems based on these ideas (see {\cite{Carlson2}}). $<_1$ is a binary
  relation in the class of ordinals and in it's original form, $\alpha <_1
  \beta$ asserts that the structure $(\alpha, <, +, <_1)$ is a
  $\Sigma_1$-substructure of $(\beta, <, +, <_1)$. Here, instead of the
  original definition of \ $<_1$, an equivalent, model-theoretical notion (see
  appendix) is consider as the fundamental notion: $\alpha <_1 \beta$ means
  $\alpha < \beta$ and the following assertion: for any finite subset $Z$ of
  $\beta$, there exists an ($<, +, <_1$)-embedding $h : Z \longrightarrow
  \alpha$ with $h \left|_{Z \cap \alpha} = \tmop{Id}_{Z \cap \alpha} \right.$
  (see definition \ref{Definition_of_<less>=_1}). Moreover, $\alpha
  \leqslant_1 \beta$ stands for $\alpha = \beta$ or $\alpha <_1 \beta$.
  
  The study of $<_1$, as done here, is then a study of (a sort of)
  isomorphisms between the finite subsets of an ordinal. In this introductory
  article we will study the (canonical) isomorphisms $\left\{ g \left( 0,
  \alpha, \beta \right) | \alpha, \beta \in \mathbbm{P} \right\}$, the
  $<^0$-relation and it's cofinality properties and see how it is that $<_1$
  induces, through all of these notions, thinner $\kappa$-club classes of
  ordinals.
  
  In comming articles it will be shown the complete generalization of these
  ideas to the thinnest $\kappa$-club classes induced by $<_1$.
\end{abstract}

\section{Basic conventions used throughout this work}

We use the standard logical and set theoretical symbols in it's standard way:
$\wedge, \vee, \Longrightarrow, \Longleftrightarrow, \forall, \exists, \neg,
\emptyset, \cup, \cap, \subset, =, \in$, etc.

By $B \subset_{\tmop{fin}} A$ we mean $B$ is a finite subset of $A$.

$h : A \longrightarrow B$ denotes that $h$ is a functional with domain $A$ and
codomain $B$.

For a functional $h : A \longrightarrow B$ and $C \subset A$, we define $h [C]
\assign \{h (x) | x \in C\}$.

For a functional $h : A \longrightarrow B$, we denote $\tmop{Dom} h \assign A$
and $\tmop{Im} h \assign h [A]$.

By $\tmop{OR}$ we denote the class of ordinals.

$0, 1, 2$,... denote, as usual, the finite ordinals.

$\omega$ denotes the first infinite ordinal.

$\tmop{Lim}$ denotes the class of limit ordinals.

$\mathbbm{P}$ denotes the class of additive principal ordinals.

$\mathbbm{E}$ denotes the class of epsilon numbers.

$<$, $+$, $\lambda x. \omega^x$ denote the usual order, the usual addition and
the usual $\omega$-base-exponentiation in the ordinals, respectively.

For an ordinal $\alpha \in \tmop{OR}$, $\varepsilon_{\alpha}$ denotes the
$\alpha$-th epsilon number.

$\min A$ denotes the minimum element of $A$ (with respect to the order $<$).

$\max A$ denotes the maximum element of $A$ (with respect to $<$).

In case $\exists \alpha \in \tmop{OR} .A \subset \alpha$, then $\sup A$
denotes the minimal upper bound of $A$ with respect to $<$ (the supremum of
$A$).

$\tmop{Lim} A \assign \tmop{Lim} (A) \assign \{\alpha \in \tmop{OR} | \alpha =
\sup (A \cap \alpha)\}$.

By $(\xi_i)_{i \in I} \subset A$ we mean $(\xi_i)_{i \in I}$ is a sequence of
elements of $A$.

Given an ordinal $\alpha \in \tmop{OR}$ and a sequence $(\xi_i)_{i \in I}
\subset \tmop{OR}$, we say that $(\xi_i)_{i \in I}$ is cofinal in $\alpha$
whenever $I \subset \tmop{OR}$, $\forall i \in I \forall j \in I.i \leqslant j
\Longrightarrow \xi_i \leqslant \xi_j$, $\forall i \in I \exists j \in I.i < j
\wedge \xi_i < \xi_j$ and $\sup \{\xi_i | i \in I\} = \alpha$. By
$\xi_i\underset{cof}{\longhookrightarrow}\alpha$ we mean that the sequence
$(\xi_i)_{i \in I}$ is cofinal in $\alpha$.

Whenever we write $\alpha =_{\tmop{CNF}} \omega^{A_1} a_1 + \ldots +
\omega^{A_n} a_n$, we mean that $\omega^{A_1} a_1 + \ldots + \omega^{A_n} a_n$
is the cantor normal form of $\alpha$, that is: $\alpha = \omega^{A_1} a_1 +
\ldots + \omega^{A_n} a_n$, $a_1, \ldots, a_n \in \omega \backslash \{0\}$,
$A_1, \ldots, A_n \in \tmop{OR}$ and $A_1 > \ldots > A_n$.

Given two ordinals $\alpha, \beta \in \tmop{OR}$ with $\alpha \leqslant
\beta$, we denote:

$[\alpha, \beta] \assign \{\sigma \in \tmop{OR} | \alpha \leqslant \sigma
\leqslant \beta\}$

$[\alpha, \beta) \assign \{\sigma \in \tmop{OR} | \alpha \leqslant \sigma <
\beta\}$

$(\alpha, \beta] \assign \{\sigma \in \tmop{OR} | \alpha < \sigma \leqslant
\beta\}$

$(\alpha, \beta) \assign \{\sigma \in \tmop{OR} | \alpha < \sigma < \beta\}$

Given $\alpha \in \mathbb{E}$, we denote by $\alpha^+$ or by $\alpha (+^1)$ to
$\min \{e \in \mathbbm{E}| \alpha < e\}$.

For a set $A$, $|A|$ denotes the cardinality of $A$; the only one exception to
this convention is when we denote as $| \tmop{ID}_n |$ and $|
\Pi^1_1$-$\tmop{CA}_0 |$ to the proof theoretic ordinals of the theories
$\tmop{ID}_n$ and $\Pi^1_1$-$\tmop{CA}_0$ respectively.

\section{The $<_1$-relation}

Our purpose is to study the (binary) relation $<_1$ defined by recursion on
the ordinals as follows

\begin{definition}
  \label{Definition_of_<less>=_1}Let $\beta \in \tmop{OR}$ be arbitrary and
  suppose $\alpha' <_1 \beta'$ has already been defined for any $\beta' \in
  \beta \cap \tmop{OR}$ and for any $\alpha' \in \tmop{OR}$. Let $\alpha \in
  \tmop{OR}$ be arbitrary.Then \\
  $\alpha <_1 \beta$ $: \Longleftrightarrow$ $\alpha < \beta$ and $\forall Z
  \subset_{\tmop{fin}} \beta \exists \tilde{Z} \subset_{\tmop{fin}} \alpha .
  \exists h$ such that:
  
  \ \ \ \ \ \ \ \ \ \ \ (i) $h : (Z, +, <, <_1) \longrightarrow ( \tilde{Z},
  +, <, <_1)$ is an isomorphism, that is:
  
  \ \ \ \ \ \ \ \ \ \ \ \ \ \ \ + $h : Z \longrightarrow \tilde{Z}$ is a
  bijection.
  
  \ \ \ \ \ \ \ \ \ \ \ \ \ \ \ + For any $a_1, a_2 \in Z$
  
  \ \ \ \ \ \ \ \ \ \ \ \ \ \ \ \ \ \ \ $\bullet$ $a_1 + a_2 \in Z
  \Longleftrightarrow h (a_1) + h (a_2) \in \tilde{Z}$
  
  \ \ \ \ \ \ \ \ \ \ \ \ \ \ \ \ \ \ \ $\bullet$ If $a_1 + a_2 \in Z$, then
  $h (a_1 + a_2) = h (a_1) + h (a_2)$.
  
  \ \ \ \ \ \ \ \ \ \ \ \ \ \ \ + For any $a_1, a_2 \in Z$,
  
  \ \ \ \ \ \ \ \ \ \ \ \ \ \ \ \ \ \ \ $\bullet$ $a_1 < a_2
  \Longleftrightarrow h (a_1) < h (a_n)$.
  
  \ \ \ \ \ \ \ \ \ \ \ \ \ \ \ \ \ \ \ $\bullet$ $a_1 <_1 a_2
  \Longleftrightarrow h (a_1) <_1 h (a_n)$.
  
  \ \ \ \ \ \ \ \ \ \ \ (ii) $h|_{Z \cap \alpha} = \tmop{Id} |_{Z \cap
  \alpha}$, where $\tmop{Id} |_{Z \cap \alpha} : Z \cap \alpha \longrightarrow
  Z \cap \alpha$ is the identity function.

  By $\alpha \leqslant_1 \beta$ we mean that $\alpha <_1 \beta$ or $\alpha =
  \beta$. Moreover, to make our notation simpler, we will write $h|_{\alpha} =
  \tmop{Id} |_{\alpha}$ instead of $h|_{Z \cap \alpha} = \tmop{Id} |_{Z \cap
  \alpha}$.
\end{definition}

\begin{remark}
  We will eventually use functions $f : Z \longrightarrow \tilde{Z}$ that are
  $\lambda x. \omega^x$-isomorphisms; of course, by this we mean the analogous
  situation as the one we had with $+$ above: \\
  For any $a \in Z$,
  
  $\bullet$ $\omega^a \in Z \Longleftrightarrow f (\omega^{\alpha}) \in
  \tilde{Z}$
  
  $\bullet$ If $\omega^a \in Z$, then $f (\omega^a) = \omega^{f (a)}$.
\end{remark}

Some of the most basic properties that $\leqslant_1$ satisfies are the
following

\begin{proposition}
  \label{connectedness_and_continuity}Let $\alpha, \beta, \gamma \in
  \tmop{OR}$.
  \begin{enumeratealpha}
    \item $\alpha \leqslant_1 \beta \Longrightarrow \{x \in \tmop{OR} | \alpha
    \leqslant_1 x \leqslant \beta\} = [\alpha, \beta]$.
    
    \item Let $(\xi_i)_{i \in I} \subset \tmop{OR}$ be a sequence such that
    $\xi_i \underset{\tmop{cof}}{\longhookrightarrow} \beta$. Then \\
    $[\forall i \in I. \alpha \leqslant_1 \xi_i] \Longrightarrow \alpha
    \leqslant_1 \beta$.
    
    \item $\alpha \leqslant_1 \beta \leqslant_1 \gamma \Longrightarrow \alpha
    \leqslant_1 \gamma$.
    
    \item Let $(\xi_i)_{i \in I} \subset \tmop{OR}$ be a sequence such that
    $\xi_i \underset{\tmop{cof}}{\longhookrightarrow} \beta$. Then \\
    $[\exists i_0 \in I. \alpha \nless_1 \xi_{i_0} \wedge \alpha < \xi_{i_0}]
    \Longrightarrow \alpha \nless_1 \beta$.
  \end{enumeratealpha}
\end{proposition}

\begin{proof}
  The proofs of $a)$, $b)$ and $c)$ follow direct from definition
  \ref{Definition_of_<less>=_1}. Moreover, $d)$ follows easily from $a)$.
\end{proof}

We call $\leqslant_1${\tmstrong{-connectedness}} (or just connectedness) to
the property $a$) of previous proposition \ref{connectedness_and_continuity};
moreover, we call $\leqslant_1${\tmstrong{-continuity}} (or just continuity)
and $\leqslant_1${\tmstrong{-transitivity}} (or just transitivity) to the
properties $b$) and $c$) (respectively) of the same proposition. We will make
use of the three of them over and over along all our work.

\begin{proposition}
  \label{existence_of_m(alpha)}Let $\alpha, \beta \in \tmop{OR}$ with $\alpha
  < \beta$ and $\alpha \nless_1 \beta$. Then there exists $\gamma \in [\alpha,
  \beta)$ such that
  \begin{enumeratealpha}
    \item $\{x \in \tmop{OR} | \alpha \leqslant_1 x\} = [\alpha, \gamma]$.
    
    \item $\{x \in \tmop{OR} | \alpha < x, \alpha \nless_1 x\} = [\gamma + 1,
    \infty)$.
    
    \item  For any $\sigma > \gamma$, $\gamma \nless_1 \sigma$.
  \end{enumeratealpha}
\end{proposition}

\begin{proof}
  Let $k \assign \min \{r \in \tmop{OR} |r > \alpha \nless_1 r\}$. Then $k
  \leqslant \beta$. Moreover, since $\forall \sigma \in [\alpha, k) . \alpha
  \leqslant_1 \sigma$, then $k$ must be a successor (otherwise, by
  $\leqslant_1$-continuity would follow $\alpha <_1 k$). So $k = \gamma + 1
  \leqslant \beta$ for some $\gamma \in \tmop{OR}$ \ and therefore $\{x \in
  \tmop{OR} | \alpha \leqslant_1 x\} = [\alpha, \gamma]$. This shows $a$).
  
  On the other hand, note that for any $\sigma \geqslant k$, it is not
  possible that $\alpha \leqslant_1 \sigma$ (otherwise, by
  $\leqslant_1$-connectedness, one gets the contradiction $\alpha <_1 k$).
  This proves $b$).
  
  Finally, observe it is not possible that for some $\sigma > \gamma$, $\gamma
  <_1 \sigma$, otherwise, from $\alpha \leqslant_1 \gamma \leqslant_1 \sigma$
  and $\leqslant_1$-transitivity follows $\alpha <_1 \sigma$, which is
  contradictory with $b)$ (because $\sigma \geqslant k = \gamma + 1$).
\end{proof}

For an ordinal $\alpha$, the ordinal $\gamma$ referred in previous proposition
\ref{existence_of_m(alpha)} will be very important for the rest of our work.
Because of that we make the following

\begin{definition}
  \label{def_max<less>=_1_reach}(The maximum $\leqslant_1$-reach of an
  ordinal). Let $\alpha \in \tmop{OR}$. We define \\
  $m (\alpha) \assign \left\{ \begin{array}{l}
    \max \{\xi \in \tmop{OR} | \alpha \leqslant_1 \xi\} \text{ \tmop{iff}
    \tmop{there} \tmop{is} }  \beta\in\tmop{OR} \text{\tmop{with} } \alpha<\beta \text{ \tmop{and} } \alpha\nless_1\beta\\
    \infty \text{ \tmop{otherwise}, \tmop{that} \tmop{is}, } \forall \beta \in
    \tmop{OR.} \alpha < \beta \Longrightarrow \alpha <_1 \beta
  \end{array} \right .$

  Note that when $m (\alpha) \in \tmop{OR}$, then it is the only one $\gamma
  \in \tmop{OR}$ satisfying $\alpha \leqslant_1 \gamma$ and $\alpha
  \nleqslant_1 \gamma + 1$. Because of this {\tmstrong{we call}} $\tmmathbf{m
  (\alpha)}$ {\tmstrong{the maximum}}
  $\tmmathbf{\leqslant_1}${\tmstrong{-reach of $\alpha$.}}
\end{definition}

\section{Characterization of the ordinals $\alpha$ such that $\alpha <_1
\alpha + 1$}

Up to this moment we do not know whether there are ordinals $\alpha, \beta$
such that $\alpha <_1 \beta$; however, in such a case, since $\alpha < \alpha
+ 1 \leqslant \beta$, then by $\leqslant_1$-connectedness we would conclude
that the relation $\alpha <_1 \alpha + 1$ must hold. This shows that the
simplest nontrivial case when we can expect that something of the form $\alpha
<_1 \beta$ holds is for $\beta = \alpha + 1$. Then, for this simplest case,
what should $\alpha$ satisfy?. The answer to this question is the purpose of
this subsection.

\begin{proposition}
  \label{prop1.02.11.2008}Let $\alpha, \beta \in \tmop{OR}$, $\alpha
  =_{\tmop{CNF}} \omega^{\alpha_1} a_1 + \ldots + \omega^{\alpha_n} a_n$, with
  $n \geqslant 2$ or $a_1 \geqslant 2$. Moreover, suppose $\alpha < \beta$.
  Then $\alpha \nless_1 \beta$.
\end{proposition}

\begin{proof}
  Case $n \geqslant 2$. \\
  Since $\alpha < \beta$, then $\{\omega^{\alpha_1} a_1, \ldots,
  \omega^{\alpha_n} a_n \} \subset \alpha \cap \beta$, but $\beta \ni
  \omega^{\alpha_1} a_1 + \ldots + \omega^{\alpha_n} a_n = \alpha \nin
  \alpha$, and so there is no $+$-isomorphism $h : Z \rightarrow \tilde{Z}$
  from $Z \assign \{\omega^{\alpha_1} a_1, \ldots, \omega^{\alpha_n} a_n,
  \alpha\} \subset_{\tmop{fin}} \beta$ in some $\tilde{Z} \subset_{\tmop{fin}}
  \alpha$ such that $h|_{\alpha} = \tmop{Id} |_{\alpha}$, since any of such
  isomorphisms should accomplish \\
  $h (\omega^{\alpha_1} a_1 + \ldots + \omega^{\alpha_n} a_n) = h
  (\omega^{\alpha_1} a_1) + \ldots + h (\omega^{\alpha_n} a_n) = \alpha \nin
  \alpha$.
  
  The same argument works for the case $n = 1, a_1 \geqslant 2$.
\end{proof}

\begin{corollary}
  \label{corollary1.02.11.2008}Let $\alpha, \beta \in \tmop{OR}$. If $\alpha
  <_1 \beta$, then $\alpha =_{\tmop{CNF}} \omega^{\gamma} \in \mathbbm{P}
  \subset \tmop{Lim}$, for some $\gamma \in \tmop{OR}, \gamma > 0$.
\end{corollary}

\begin{proof}
  Direct from previous proposition \ref{prop1.02.11.2008}. The only left cases
  are $\alpha = 0$ or $\alpha = 1$ but for those cases it is very easy to see
  that $\alpha \nless_1 \alpha + 1$, since $\alpha + 1$ has $\alpha + 1$
  elements and $\alpha$ has only $\alpha$ elements, and so for those cases
  $\alpha \nless_1 \beta$ for any $\beta > \alpha$.
\end{proof}

\begin{proposition}
  If $\alpha = \omega^n$, $n \in \omega$, then $\alpha \nless_1 \alpha + 1$.
\end{proposition}

\begin{proof}
  Not hard. But we will give a more general proof of this fact in the next
  propositions.
\end{proof}

\begin{corollary}
  Let $\alpha, \beta \in \tmop{OR}$. If $\alpha <_1 \beta$, then $\alpha
  =_{\tmop{CNF}} \omega^{\gamma}$ for some $\gamma \in \tmop{OR}, \gamma
  \geqslant \omega$.
\end{corollary}

\begin{proof}
  From previous proposition and previous corollary. (This will be proved in
  the next three propositions in a more general way).
\end{proof}

\begin{proposition}
  \label{prop.exist.maximum}Let $\alpha \in \tmop{OR}$, $1 < \alpha \in
  \tmop{Lim}$. Suppose $\alpha \cap \mathbbm{P}$ is not confinal in $\alpha$. Then
  $M \assign \max (\mathbbm{P} \cap \alpha)$ exists.
\end{proposition}

\begin{proof}
  Since $\mathbbm{P}$ is a closed class of ordinals, then $\sup (\mathbbm{P}
  \cap \alpha) \in \mathbbm{P} \cap \alpha$. So $M = \sup (\mathbbm{P} \cap
  \alpha)$.
\end{proof}

\begin{proposition}
  \label{prop0.03.11.2008}Let $\alpha, p \in \tmop{OR}$, $1 < \alpha <_1 p +
  1$, with $p \in \mathbbm{P}$ an additive principal number. Then:
  
  (i) $\alpha \cap \mathbbm{P}$ is confinal in $\alpha$.
  
  (ii) $\alpha \in \tmop{Lim} \mathbbm{P} \subset \mathbbm{P}$, (or
  equivalently, (ii') $\alpha = \omega^{\gamma}$, for $\gamma \in
  \tmop{Lim}$.)
\end{proposition}

\begin{proof}
  $(i)$. By corollary \ref{corollary1.02.11.2008} we know $\alpha \in
  \tmop{Lim}$. Now, suppose $\alpha \cap \mathbbm{P}$ is not confinal in
  $\alpha$. Then by previous proposition \ref{prop.exist.maximum}, let $M
  \assign \max \alpha \cap \mathbbm{P} \in \alpha$.
  
  Then $M + p = p$, but on the other hand, $\forall \gamma \in \alpha .M +
  \gamma > \gamma$. Therefore, for \\
  $Z \assign \{M, p\} \subset_{\tmop{fin}} p + 1$ and for any $\tilde{Z}
  \subset \alpha$ there is no $+$-isomorphism $h : Z \rightarrow \tilde{Z}$,
  such that $h|_{\alpha} = \tmop{Id} |_{\alpha}$, since any such function
  would satisfy \\
  $h (p) = h (M + p) = h (M) + h (p) = M + h (p) > h (p)$ (Contradiction!).
  
  Thus $\alpha \cap \mathbbm{P}$ is confinal in $\alpha$.
  
  $(i i)$. Clear from $(i)$.
\end{proof}

\begin{corollary}
  \label{alpha<less>_1beta_implies_alpha_in_Lim(P)}Let $\alpha, \beta \in
  \tmop{OR}$ such that $\alpha <_1 \beta$. Then $\alpha \in \tmop{Lim}
  \mathbbm{P}$.
\end{corollary}

\begin{proof}
  From corollary \ref{corollary1.02.11.2008} we have that $\alpha <_1 \beta$
  implies $\alpha \in \mathbbm{P}$. Moreover, from $\alpha <_1 \beta$ we know
  $\alpha < \alpha + 1 \leqslant \beta$ and then $\alpha <_1 \alpha + 1$ by
  $<_1$-connectedness. Finally, from $\alpha <_1 \alpha + 1$, $\alpha \in
  \mathbbm{P}$ and the previous proposition \ref{prop0.03.11.2008}, $\alpha
  \in \tmop{Lim} \mathbbm{P}$.
\end{proof}

\begin{proposition}
  \label{characterization_of_alpha<less>_1alpha+1}Let $\alpha \in \tmop{OR}$.
  The following are equivalent:
  \begin{enumeratealpha}
    \item $\alpha <_1 \alpha + 1$
    
    \item $\alpha \in \tmop{Lim} \mathbbm{P}$
    
    \item $\alpha = \omega^{\gamma}$ for some $\gamma \in \tmop{Lim}$.
    
    \item $\alpha = \omega^{\gamma}$ and $\gamma =_{\tmop{CNF}} \omega^{A_1}
    a_1 + \ldots + \omega^{A_n} a_n$ with $A_n \neq 0$.
  \end{enumeratealpha}
\end{proposition}

\begin{proof}
  The proof of $b) \Longleftrightarrow c) \Longleftrightarrow d)$ is a
  standard fact about ordinals.

  $a) \Longrightarrow b)$ is previous corollary
  \ref{alpha<less>_1beta_implies_alpha_in_Lim(P)}.

  So let's prove $b) \Longrightarrow a)$.
  
  Let $\alpha \in \tmop{Lim} \mathbbm{P}$. Take $B \subset_{\tmop{fin}} \alpha
  + 1$. If $\alpha \nin B$, then $l : B \longrightarrow \alpha$, $l (x)
  \assign x$ is an \\
  $(<, <_1, +)$-isomorphism such that $l|_{\alpha} = \tmop{Id}_{\alpha}$. So
  suppose $B = \{a_0 < \ldots < a_n = \alpha\}$ for some natural number $n$.
  Let $A \assign \{m (a) |a \in (B \cap \alpha) \wedge m (a) < \alpha\}$.
  Since $\alpha \in \tmop{Lim} \mathbbm{P}$ and $A$ is finite, then there
  exists $\rho \in (a_{n - 1}, \alpha) \cap (\max A, \alpha) \cap
  \mathbbm{P}$. Let $h : B \longrightarrow h [B] \subset \alpha$ be the
  function \\
  $h (x) \assign \left\{ \begin{array}{l}
    x \text{ \tmop{iff} } x < \alpha\\
    \rho \text{ \tmop{otherwise}}
  \end{array} \right.$. It is clear that $h|_{\alpha} = \tmop{Id}_{\alpha}$.

  We assure that $h$ is an $(<, <_1, +)$-isomorphism.
  
  The details are left to the reader.
\end{proof}

\section{The ordinals $\alpha$ satisfying $\alpha <_1 t$, for some $t \in
[\alpha, \alpha \omega)$.}

We have seen previously that the ``solutions of the $<_1$-inequality'' $x <_1
x + 1$ are the elements of $\tmop{Lim} \mathbbm{P}$. It is natural then to ask
himself about the solutions of $x <_1 x + 2$ or of $x <_1 x + \omega$. In
general, this question can be informally stated as: What are the solutions of
$x <_1 \beta$, where ``we pick $\beta$ as big as we can''?. The descriptions
of such solutions in a certain way is a main purpose of this work: we will
describe them as certain classes of ordinals obtained by certain thinning
procedure. The rest of this article is devoted to our investigations
concerning this question for $x \in \mathbbm{P}$ and $\beta \in [x, x
\omega]$. We will introduce various concepts that at the first sight may look
somewhat artificial; however, these concepts and the way to use them is just
``the most basic realization'' of the general tools and methodology shown in
comming articles that will allow us to understand the $<_1$-relation in the
whole class of ordinals.

\subsection{Class(0)}

\begin{definition}
  Let $\tmop{Class} (0) \assign \mathbbm{P}$.
\end{definition}

\begin{definition}
  For $\alpha, \beta \in \tmop{OR}$, let \\
  $- \alpha + \beta \assign \left\{ \begin{array}{l}
    \text{\tmop{the} \tmop{only} \tmop{one} \tmop{ordinal} } \sigma
    \text{ \tmop{such} \tmop{that} } \alpha + \sigma = \beta \text{ \tmop{iff} }
    \alpha \leqslant \beta\\
    \\
    - 1 \text{ \tmop{otherwise}}
  \end{array} \right.$
\end{definition}

\begin{definition}
  Let $\alpha, c \in \tmop{Class} (0)$ with $\alpha \leqslant c$.

  We define $g (0, \alpha, c) : \alpha \omega \longrightarrow c \omega$ as:\\
  $g (0, \alpha, c) (x) \assign x$ iff $x < \alpha$.\\
  $g (0, \alpha, c) (x) \assign c n + l$ iff $x \in [\alpha n, \alpha n +
  \alpha) \wedge x = \alpha + l$ for some $l \in \alpha$.

  Moreover, we define $g (0, c, \alpha) \assign g (0, \alpha, c)^{- 1}$.
\end{definition}

\begin{proposition}
  \label{g(0,a,c)_prop1}Let $\alpha, c \in \tmop{Class} (0)$. Then
  \begin{enumeratenumeric}
    \item $\tmop{Dom} g (0, \alpha, c) = (\alpha \cap c) \cup \bigcup_{n \in
    [1, \omega)} \{t \in [\alpha n, \alpha n + \alpha) | - \alpha n + t <
    c\}$.
    
    \item $\tmop{Im} g (0, \alpha, c) = (\alpha \cap c) \cap \bigcup_{n \in
    [1, \omega)} \{t \in [c n, c n + c) | - c n + t < \alpha\}$.
    
    \item $g (0, \alpha, c) : \tmop{Dom} g (0, \alpha, c) \longrightarrow
    \tmop{Im} g (0, \alpha, c)$ is an $(<, +)$-isomorphism and $g (0, \alpha,
    c) |_{\alpha} = \tmop{Id}_{\alpha}$.
  \end{enumeratenumeric}
\end{proposition}

\begin{proof}
  Left to the reader.
\end{proof}

\begin{proposition}
  \label{g(0,a,c)|_(a,a2)_is_iso}Let $\alpha, c \in \tmop{Class} (0)$ and $X
  \assign (\alpha \cap c) \cup \bigcup_{n \in [1, \omega)} \{t \in [\alpha n,
  \alpha n + \alpha) | - \alpha n + t < c\}$. Then the function $H : (\alpha,
  \alpha \omega) \cap X \longrightarrow H [(\alpha, \alpha \omega) \cap X]
  \subset (c, c \omega)$, $H (x) \assign g (0, \alpha, c) (x)$ is an \\
  $(<, <_1, +)$-isomorphism.
\end{proposition}

\begin{proof}
  Let $\alpha, c$, $X$ and $H$ be as stated. By previous proposition
  \ref{g(0,a,c)_prop1} follows easily that $H$ is an $(<, +)$-isomorphism.
  Moreover, $H$ is also an $<_1$-isomorphism because by proposition
  \ref{characterization_of_alpha<less>_1alpha+1} and $<_1$-connectedness it
  follows that $\forall a, b \in (\alpha, \alpha \omega) .a \nless_1 b$ and
  $\forall a, b \in (c, c \omega) .a \nless_1 b$.
\end{proof}

\begin{definition}
  \label{def_T(0,alpha,t)}Consider $\alpha \in \tmop{Class} (0)$ and $t \in
  \alpha \omega$. \\
  We define $T (0, \alpha, t) \assign \left\{ \begin{array}{l}
    \{t\} \text{ \tmop{iff} } t < \alpha\\
    \\
    \{t, - \alpha n + t\} \text{ \tmop{iff} } t \in [\alpha n, \alpha n +
    \alpha) \text{\tmop{for} \tmop{some}} n \in [1, \omega) .
  \end{array} \right.$
\end{definition}

\begin{proposition}
  \label{Domg(0,a,t)_and_T(0,a,t)}$\forall \alpha, c \in \tmop{Class} (0) .
  \forall t \in \alpha \omega .t \in \tmop{Dom} (g (0, \alpha, c))
  \Longleftrightarrow T (0, \alpha, t) \cap \alpha \subset c$
\end{proposition}

\begin{proof}
  Direct from definition \ref{def_T(0,alpha,t)} and proposition
  \ref{g(0,a,c)_prop1}.
\end{proof}

\begin{definition}
  Let $\alpha \in \tmop{Class} (0)$ and $t \in [\alpha, \alpha \omega]$. By
  $\alpha <^0 t$ we mean
  \begin{enumeratenumeric}
    \item $\alpha < t$
    
    \item $\forall B \subset_{\tmop{fin}} t. \exists \delta \in \tmop{Class}
    (0) \cap \alpha$ such that
    \begin{enumerateroman}
      \item $( \bigcup_{t \in B} T (0, \alpha, t) \cap \alpha) \subset
      \delta$;
      
      \item The function $h : B \longrightarrow h [B]$ defined as $h (x)
      \assign g (0, \alpha, \delta) (x)$ is an $(<, <_1, +)$-isomorphism with
      $h|_{\alpha} = \tmop{Id}_{\alpha}$.
    \end{enumerateroman}
  \end{enumeratenumeric}
  As usual, $\alpha \leqslant^0$ just means $\alpha <^0 t$ or $\alpha = t$.
\end{definition}

\begin{proposition}
  Let $\alpha \in \tmop{Class} (0)$, $(\xi_i)_{i \in I} \subset [\alpha,
  \alpha \omega] \ni \beta, \gamma$. Then
  \begin{enumeratenumeric}
    \item $\alpha \leqslant^0 \beta \Longrightarrow \alpha \leqslant_1 \beta$.
    
    \item If $\alpha \leqslant \beta \leqslant \gamma \wedge \alpha
    \leqslant^0 \gamma$ then $\alpha \leqslant^0 \beta$. \ \ \ \ \ \ \ \ \ \ \
    \ \ ($\leqslant^0$-connectedness)
    
    \item If $\forall i \in I. \alpha \leqslant^0 \xi_i \wedge \xi_i
    \underset{\tmop{cof}}{\longhookrightarrow} \beta$ then $\alpha \leqslant^0
    \beta$. \ \ \ \ \ \ \ ($\leqslant^0$-continuity)
  \end{enumeratenumeric}
\end{proposition}

\begin{proof}
  Left to the reader.
\end{proof}

\begin{proposition}
  \label{<less>^0_1st_cof._prop.}(First fundamental cofinality property of
  $<^0$). \\
  Let $\alpha \in \tmop{Class} (0)$ and $t \in [\alpha, \alpha \omega)$.\\
  Then $\alpha <^0 t + 1 \Longrightarrow \alpha \in \tmop{Lim} \{\beta \in
  \tmop{Class} (0) | T (0, \alpha, t) \cap \alpha \subset \beta \wedge \beta
  \leqslant_1 g (0, \alpha, \beta) (t)\}$.
\end{proposition}

\begin{proof}
  Let $\alpha$, $t$ be as stated.

  Suppose $\alpha <^0 t + 1$. \ \ \ \ \ \ \ {\tmstrong{(*1)}}

  Let $\gamma \in \alpha$ be arbitrary and consider $B_{\gamma} \assign
  \{\gamma, \alpha, t\} \subset_{\tmop{fin}} t + 1$. By (*1) there exists
  $\delta_{\gamma} \in \alpha \cap \tmop{Class} (0)$ such that $( \bigcup_{q
  \in B} T (0, \alpha, q) \cap \alpha) \subset \delta_{\gamma}$ and the
  function $h : B \longrightarrow h [B] \subset \alpha$, $h (x) \assign g (0,
  \alpha, \delta_{\gamma}) (x)$ is an ($<, <_1, +$)-isomorphism with
  $h|_{\alpha} = \tmop{Id}_{\alpha}$. In particular, note:
  
  1. $\gamma < \delta_{\gamma}$ because $\gamma \in ( \bigcup_{q \in B} T (0,
  \alpha, q) \cap \alpha) \subset \delta_{\gamma}$.
  
  2. $\delta_{\gamma} = g (0, \alpha, \delta_{\gamma}) (\alpha) \leqslant_1 g
  (0, \alpha, \delta_{\gamma}) (t)$ because $T (0, \alpha, t) \cap \alpha
  \subset \delta_{\gamma}$ and $\alpha \leqslant_1 t \Longleftrightarrow h
  (\alpha) \leqslant_1 h (t)$.

  Since the previous was done for arbitrary $\gamma < \alpha$, 1 and 2 show
  that \\
  $\forall \gamma \in \alpha \exists \delta_{\gamma} \in \{\beta \in
  \tmop{Class} (0) | \gamma < \beta \wedge T (0, \alpha, t) \cap \alpha
  \subset \beta \wedge \beta \leqslant_1 g (0, \alpha, \beta) (t)\}$. Thus \\
  $\alpha \in \tmop{Lim} \{\beta \in \tmop{Class} (0) | T (0, \alpha, t) \cap
  \alpha \subset \beta \wedge \beta \leqslant_1 g (0, \alpha, \beta) (t)\}$.
\end{proof}

\begin{proposition}
  \label{<less>^0_2nd_cof._prop.}(Second fundamental cofinality property of
  $<^0$). \\
  Let $\alpha \in \tmop{Class} (0)$ and $t \in [\alpha, \alpha \omega)$.\\
  Then $\alpha <^0 t + 1 \Longleftarrow \alpha \in \tmop{Lim} \{\beta \in
  \tmop{Class} (0) | T (0, \alpha, t) \cap \alpha \subset \beta \wedge \beta
  \leqslant_1 g (0, \alpha, \beta) (t)\}$.
\end{proposition}

\begin{proof}
  Let $\alpha$, $t$ be as stated.

  Suppose  $\alpha \in \tmop{Lim} \{\beta \in \tmop{Class} (0) | T (0, \alpha,
  t) \cap \alpha \subset \beta \wedge \beta \leqslant_1 g (0, \alpha, \beta)
  (t)\}$. \ \ {\tmstrong{(*1)}}

  We prove by induction: $\forall s \in [\alpha, t + 1] . \alpha \leqslant^0
  s$. \ \ \ \ \ \ \ {\tmstrong{(*2)}}

  Let $s \in [\alpha, t + 1]$ and suppose $\forall q \in s \cap [\alpha, t +
  1] . \alpha \leqslant^0 q$. \ \ \ \ \ \ \ {\tmstrong{(IH)}}

  Case $s = \alpha$.
  
  Then clearly (*2) holds.

  Case $s \in \tmop{Lim} \cap (\alpha, t + 1]$.
  
  Since by our (IH) $\forall q \in s \cap [\alpha, t + 1] . \alpha \leqslant^0
  q$, then $\alpha \leqslant^0 s$ follows by $\leqslant^0$-continuity.

  Suppose $s = l + 1 \in (\alpha, t + 1]$.
  
  Let $B \subset_{\tmop{fin}} l + 1$ be arbitrary. Consider $A \assign
  \{\alpha, l\} \cup \{m (a) | a \in B \cap \alpha \wedge m (a) < \alpha\}$.
  Then the set $\bigcup_{q \in B \cup A} T (0, \alpha, q) \cap \alpha$ is
  finite and then, by (*1), there is some $\delta \in \tmop{Class} (0) \cap
  \alpha$ such that $( \bigcup_{q \in B \cup A} T (0, \alpha, q) \cap \alpha)
  \subset \delta \wedge \delta \leqslant_1 g (0, \alpha, \delta) (t)$. \ \ {\tmstrong{(*3)}}

  Consider the function $h : B \longrightarrow h [B] \subset \alpha$ defined
  as $h (x) \assign g (0, \alpha, \delta) (x)$. From (*3) and propositions
  \ref{Domg(0,a,t)_and_T(0,a,t)} we know that $h$ is well defined; moreover,
  from proposition \ref{g(0,a,c)_prop1} it follows that $h$ is an ($<,
  +$)-isomorphism with $h|_{\alpha} = \tmop{Id}_{\alpha}$. \ \ \ \ \ \ \
  {\tmstrong{(*4)}}

  Before showing that $h$ is an $<_1$-isomorphism, we do two observations:

  Let $b \in B$ with $b \geqslant \alpha$. Then $\alpha \leqslant b \leqslant
  l$, which, together with $\alpha\underset{\text{\tmop{by (IH)}}}{\leqslant^0}l$, imply
  by $\leqslant^0$-connectedness that $\alpha \leqslant^0 b$; subsequently,
  $\alpha \leqslant_1 b$. This shows $\forall b \in B. \alpha \leqslant b
  \Longrightarrow \alpha \leqslant_1 b$ \ \ \ \ \ \ \ {\tmstrong{(*5)}}

  Let $b \in B$ with $b \geqslant \alpha$. Then $\alpha \leqslant b \leqslant
  t$ implies \\
  $\delta = g(0, \alpha, \delta) (\alpha)\underset{g(0, \alpha, \delta)\text{ strictly increasing}
  }{\leqslant}g(0, \alpha, \delta)(b)$\\
  \hspace*{16ex}
  $\underset{g (0,\alpha,\delta) \text{ strictly increasing}}{\leqslant}g (0, \alpha, \delta)(t)$;
  the latter together with \\ 
  $\delta\underset{\text{ by (*3)}}{<_1}$ $g (0,
  \alpha, \delta) (t)$ imply by $\leqslant_1$-connectedness that \\
  $g (0, \alpha, \delta) (\alpha) = \delta \leqslant_1 g (0, \alpha, \delta)
  (b)$. All this shows $\forall b \in B. \alpha \leqslant b \Longrightarrow
  \delta \leqslant_1 g (0, \alpha, \delta) (b)$ \ \ \ \ \ \ \
  {\tmstrong{(*6)}}.

  Now we show that $h$ is an $<_1$-isomorphism. \ \ \ \ \ \ \
  {\tmstrong{(*7)}}

  Let $a, b \in B$ with $a < b$.
  
  {\tmstrong{Case}} $\tmmathbf{\alpha < a < b}$.
  
  Then $a <_1 b$ $\underset{\text{by proposition }
  \ref{g(0,a,c)|_(a,a2)_is_iso}}{\Longleftrightarrow}$ $h (a) = g (0, \alpha,
  \delta) (a) <_1 g (0, \alpha, \delta) (b) = h (b)$.
  
  {\tmstrong{Case}} $\tmmathbf{a = \alpha < b}$.
  
  By (*5) and (*6) we have that $\alpha <_1 b$ and $h (\alpha) = g (0,
  \alpha, \delta) (\alpha) = \delta <_1 g (0, \alpha, \delta) (b) = h (b)$.
  
  {\tmstrong{Case}} $\tmmathbf{a, b < \alpha}$.
  
  Then $a <_1 b$ $\underset{\text{by (*4)}}{\Longleftrightarrow}$ $a = h (a) <_1 b = h
  (b)$.
  
  {\tmstrong{Case}} $\tmmathbf{a < \alpha \leqslant b}$.
  \begin{itemizedot}
    \item $a <_1 b$ $\underset{\text{by } \leqslant_1\text{-connectedness and
    (*5)}}{\Longrightarrow}$ $a <_1 \alpha \leqslant_1 b$ $\underset{\text{by
    proposition }\ref{g(0,a,c)_prop1} \text{ and by (*6)}}{\Longrightarrow}$ \\
    $a = g (0, \alpha, \delta) (a) < g (0, \alpha, \delta) (\alpha) = \delta <
    \alpha \wedge a <_1 \alpha \wedge \delta \leqslant_1 g (0, \alpha, \delta)
    (b)$ $\underset{\text{by } \leqslant_1\text{-connectedness}}{\Longrightarrow}$ \\
    $a = g (0, \alpha, \delta) (a) <_1 g (0, \alpha, \delta) (\alpha) = \delta
    \wedge \delta \leqslant_1 g (0, \alpha, \delta) (b)$ $\underset{\text{by }
    \leqslant_1\text{-transitivity} }{\Longrightarrow}$\\
    $h (a) = g (0, \alpha, \delta) (a) <_1 g (0, \alpha, \delta) (b) = h (b)$.
    
    \item $a \nless_1 b \Longrightarrow a \nless_1 \alpha$ (because $a <_1
    \alpha$ implies, using (*5), that $a <_1 b$), that is, $a \in B \cap
    \alpha$ with $m (a) < \alpha$. Then, $m (a)$ $\underset{\text{by (*3)}}{<}$ $\delta
    = g (0, \alpha, \delta) (\alpha)$ $\underset{g (0, \alpha, \delta) \text{ is
    strictly increasing}}{\leqslant}$ $g (0, \alpha, \delta) (b)$, that is, $h
    (\alpha) = a \nless_1 g (0, \alpha, \delta) (b) = h (b)$.
  \end{itemizedot}

  The previous shows that (*7) holds. In fact, (4*) and (7*) show that (2*)
  also holds for the case $s = l + 1 \subset (\alpha, t + 1]$ and with this we
  have concluded the proof of (*2). Hence, the proposition holds.
\end{proof}

The idea now is that $<_1$ and $<^0$ have something to do with each other. The
relation between $<_1$ and $<^0$ is very direct (see next proposition
\ref{prop._<less>^0<less>==<gtr><less>_1}); however, when we introduce
$\tmop{Class} (1)$ (or in general $\tmop{Class} (n)$ for $n \in [1, n]$), the
way to relate $<_1$ with a relation $<^1$ (or in general $<^n$ for $n \in [1,
n]$) will be much harder and will be done through the covering theorem. So,
said in other words, the covering theorem for $\tmop{Class} (0)$ is trivial
and therefore we can prove the next proposition
\ref{prop._<less>^0<less>==<gtr><less>_1} without anymore preparations.

\begin{proposition}
  \label{prop._<less>^0<less>==<gtr><less>_1}Let $\alpha \in \tmop{Class} (0)$
  and $t \in [\alpha, \alpha \omega)$. Then $\alpha <^0 t + 1
  \Longleftrightarrow \alpha <_1 t + 1$
\end{proposition}

\begin{proof}
  \\
  $\Longrightarrow )$. Clear by the definition of $<^0$.

  $\Longleftarrow )$. Suppose $\alpha <_1 t + 1$. \ \ \ \ \ \ \
  {\tmstrong{(*1)}}

  Note (*1) and proposition \ref{characterization_of_alpha<less>_1alpha+1}
  imply that $\alpha \in \tmop{Lim} \mathbbm{P}$ \ \ \ \ \ \ \
  {\tmstrong{(*2)}}.

  Case $t = \alpha$.
  
  Let $B \subset_{\tmop{fin}} t + 1 = \alpha + 1$ be arbitrary. Since $B \cap
  \alpha$ is finite and (2*) holds, then there exists $\delta \in \mathbbm{P}$
  such that $B \cap \alpha \subset \delta$. This way, note \\
  $( \bigcup_{t \in B} T (0, \alpha, t) \cap \alpha) \subset B \cap \alpha
  \subset \beta$, and then, by proposition \ref{Domg(0,a,t)_and_T(0,a,t)}, the
  function $h : B \longrightarrow h [B] \subset \alpha$, $h (x) \assign g (0,
  \alpha, \delta) (x)$ is well defined. Finally, note that from propositions
  \ref{g(0,a,c)_prop1} and \ref{g(0,a,c)|_(a,a2)_is_iso} it follows that the
  function $h$ is an ($<, <_1, +$)-isomorphism with $h|_{\alpha} =
  \tmop{Id}_{\alpha}$.

  Case $t > \alpha$.
  
  Let $B \subset_{\tmop{fin}} t + 1$ be arbitrary. Consider \\
  $C \assign B \cup \{\alpha, 1, \alpha + 1\} \cup \{\alpha m, l, \alpha m + l
  | \alpha n + l \in B \wedge m \in [1, n] \wedge l \in [0, \alpha)\}
  \subset_{\tmop{fin}} t + 1$. So, by (*1), there exists $k : C
  \longrightarrow k [C] \subset \alpha$ an ($<, <_1, +$)-isomorphism with
  $k|_{\alpha} = \tmop{Id}_{\alpha}$. \ \ \ \ \ \ \ {\tmstrong{(*3)}} Then:
  
  1. $\alpha <_1 \alpha + 1 \Longleftrightarrow k (\alpha) <_1 k (\alpha + 1)
  = k (\alpha) + k (1) = k (\alpha) + 1$, that is, \\
  $k(\alpha)$ $\underset{\text{proposition }
  \ref{characterization_of_alpha<less>_1alpha+1}}{\in}$ $\tmop{Lim}
  \mathbbm{P}$.
  
  2. $\forall s \in C \cap \alpha .s < \alpha \Longleftrightarrow s = k (s) <
  k (\alpha)$
  
  3. $\forall n \in [1, \omega) \forall s \in C \cap [\alpha n, \alpha n +
  \alpha) . - \alpha n + s < \alpha \Longleftrightarrow - \alpha n + s = k (-
  \alpha n + s) < k (\alpha) \}$

  From 1, 2 and 3 follows that $\delta \assign k (\alpha) \in \tmop{Class} (0)
  \cap \alpha$, $( \bigcup_{t \in C} T (0, \alpha, t) \cap
  \alpha)$ $\underset{\text{propositions } \ref{g(0,a,c)_prop1} \text{ and }
  \ref{Domg(0,a,t)_and_T(0,a,t)}}{\subset}$ $\delta$ and that the function $H
  : C \longrightarrow H [C] \subset \alpha$, $H (x) \assign g (0, \alpha,
  \delta) (x)$ is well defined. Moreover, by propositions \ref{g(0,a,c)_prop1}
  it follows that $H$ is an $(<, +)$-isomorphism with $H|_{\alpha} =
  \tmop{Id}_{\alpha}$. \ \ \ \ \ \ \ {\tmstrong{(*4)}}

  Now we show that $H$ is also an $<_1$-isomorphism. \ \ \ \ \ \ \
  {\tmstrong{(*5)}}
  
  Let $a, b \in C$ with $a < b$.
  
  {\tmstrong{Case}} $\tmmathbf{a = \alpha \wedge b \in [\alpha n, \alpha n +
  \alpha)} {\tmstrong{for some}}$ $\tmmathbf{n \in [1, \omega)}{\tmstrong{.}}$
  Then $\alpha <_1 t + 1$ and $\alpha < b < t + 1$ imply by
  $\leqslant_1$-connectedness that $\alpha <_1 b$. \\
  On the other hand, note $H (\alpha) = k (\alpha)$ $\underset{\text{by (*3)}}{<_1}$ $k
  (b) = k (\alpha n + ( - \alpha n + b))$ $\underset{\text{by (*3)}}{=}$\\
  $k (\alpha n) + k (- \alpha n + b) = k (\alpha) n + (- \alpha n + b) = H
  (\alpha) n + H (- \alpha n + b)$ $\underset{\text{by (*4)}}{=}$\\
  $H (\alpha n) + H (- \alpha n + b)$ $\underset{\text{by (*4)}}{=}$ $H (\alpha n + ( -
  \alpha n + b)) = H (b)$. \ \ \ \ \ \ \ {\tmstrong{(*6)}}
  
  {\tmstrong{Case}} $\tmmathbf{a, b < \alpha}$. Then $a <_1 b
  \Longleftrightarrow a = H (a) <_1 b = H (b)$.
  
  {\tmstrong{Case}} $\tmmathbf{a < \alpha \leqslant b}$. Then $a <_1
  b$ $\underset{\leqslant_1\text{-connectedness and }
  \leqslant_1\text{-transitivity}}{\Longleftrightarrow}$ $a <_1 \alpha \leqslant_1
  b$ $\underset{\text{by (*3) and (*6)}}{\Longleftrightarrow}$
  $a = H (a) = k (a) <_1 k (\alpha) = H (\alpha) \leqslant_1 k (b) = H (b)$.
  
  {\tmstrong{Case}} $\tmmathbf{\alpha < a < b}$. Then $a <_1 b$ $\underset{\text{by
  proposition } \ref{g(0,a,c)|_(a,a2)_is_iso}}{\Longleftrightarrow}$ $H (\alpha)
  <_1 H (b)$.

  The previous shows that (*5) holds.

  Finally, from (*4), (*5) and the fact that $B \subset C$ we conclude, by
  proposition \ref{iso.restriction} in the appendices section, that the
  function $H|_B : B \longrightarrow H|_B [B] \subset \alpha$, $H|_B (x) = g
  (0, \alpha, \delta) (x)$ is an \\
  ($<, <_1, +$)-isomorphism with $H|_{\alpha} = \tmop{Id}_{\alpha}$.

  All the previous shows that $\alpha <^0 t + 1$.
\end{proof}

\begin{corollary}
  \label{cor_<less>^0_equivalences}Let $\alpha \in \tmop{Class} (0)$ and $t
  \in [\alpha, \alpha \omega)$. The following are equivalent:
  \begin{enumeratenumeric}
    \item $\alpha <^0 t + 1$
    
    \item $\alpha <_1 t + 1$
    
    \item $\alpha \in \tmop{Lim} \{\beta \in \tmop{Class} (0) | T (0, \alpha,
    t) \cap \alpha \subset \beta \wedge \beta \leqslant_1 g (0, \alpha, \beta)
    (t)\}$
  \end{enumeratenumeric}
\end{corollary}

\begin{proof}
  Direct from previous propositions \ref{prop._<less>^0<less>==<gtr><less>_1},
  \ref{<less>^0_1st_cof._prop.} and \ref{<less>^0_2nd_cof._prop.}.
\end{proof}

\subsection{A hierarchy induced by $<_1$ and the intervals $[\omega^{\gamma},
\omega^{\gamma + 1})$.}

In this subsection we show theorem \ref{Class(0)_Hierarchy} which is our way
to link ``solutions of the conditions $\alpha <_1 t + 1$, with $\alpha \in
\tmop{Class} (0)$ and $t \in [\alpha, \alpha \omega)$'' (what below is defined
as the $G^0 (t)$ sets) with a thinning procedure (the sets $A^0 (t)$, also
defined below). After that, we will see that, for $\alpha = \kappa$ a regular
non-countable ordinal, the set of ``solutions of the condition $\kappa <_1 t +
1$'' is club in $\kappa$.

\begin{definition}
  By recursion on $([\omega, \infty), <)$, we define \\
  $A^0 : [\omega, \infty)
  \longrightarrow \tmop{Subclasses} (\tmop{OR})$ in the following way: Let $t
  \in [\omega, \infty)$ be arbitrary. Let $\alpha \in \tmop{Class} (0)$ be
  such that $t \in [\alpha, \alpha \omega)$. Then \\
  $A^0 (t) \assign \left\{
  \begin{array}{l}
    (\tmop{LimClass} (0)) \cap (\alpha + 1)  \text{ \tmop{iff} } t = \alpha\\
    \tmop{Lim} A^0 (l + 1) \text{ \tmop{iff} } t = l + 1\\
    \tmop{Lim} \{r \in \tmop{Class} (0) \cap (\alpha + 1) | T (0, \alpha, t)
    \cap \alpha \subset r \wedge \\
    \hspace*{5ex} r \in \bigcap_{s \in \{q \in (\alpha, t) | T (0,
    \alpha, q) \cap \alpha \subset r\}} A^0 (s)\} \text{ \tmop{iff} } t \in
    [\alpha, \alpha \omega) \cap \tmop{Lim}
  \end{array} \right.$
  
  \hspace*{3ex} $= \left\{ \begin{array}{l}
    (\tmop{LimClass} (0)) \cap (\alpha + 1) \text{ \tmop{iff} } t = \alpha\\
    \tmop{Lim} A^0 (l + 1) \text{ \tmop{iff} } t = l + 1\\
    \tmop{Lim} \{r \in \tmop{Class} (0) \cap (\alpha + 1) |- \alpha n + t < r \wedge \\
    r \in \bigcap_{s \in \{q \in (\alpha, t) | T (0, \alpha, q) \cap \alpha
    \subset r\}} A^0 (s)\} \text{ \tmop{iff} } \left\{ \begin{array}{l}
      t \in [\alpha n, \alpha n + \alpha) \cap \tmop{Lim}\\
      \text{\tmop{for} \tmop{some} } n \in [1, \omega)
    \end{array} \right.
  \end{array} \right.$

  On the other hand, we define $G^0 : [\omega, \infty) \longrightarrow
  \tmop{Subclasses} (\tmop{OR})$ as follows: Let $t \in [\omega, \infty)$ be
  arbitrary. Let $\alpha \in \tmop{Class} (0)$ and $n \in [1, \omega)$ be such
  that $t \in [\alpha n, \alpha n + \alpha)$. Then \\
  $G^0 (t) \assign \{\beta \in \tmop{Class} (0) | T (0, \alpha, t) \cap \alpha
  \subset \beta \leqslant \alpha \wedge \beta \leqslant^0 g (0, \alpha, \beta)
  (t) + 1\}$
  
  \ \ \ \ $= \{\beta \in \tmop{Class} (0) | - \alpha n + t < \beta \leqslant
  \alpha \wedge \beta \leqslant^0 g (0, \alpha, \beta) (t) + 1\}$
  
  \ \ \ \ $=$, by proposition \ref{prop._<less>^0<less>==<gtr><less>_1},
  
  \ \ \ \ $= \{\beta \in \tmop{Class} (0) | - \alpha n + t < \beta \leqslant
  \alpha \wedge \beta \leqslant_1 g (0, \alpha, \beta) (t) + 1\}$.
  
  \ \ \ \ $= \{\beta \in \tmop{Class} (0) | T (0, \alpha, t) \cap \alpha \subset
  \beta \leqslant \alpha \wedge \beta \leqslant_1 g (0, \alpha, \beta) (t) +
  1\}$
\end{definition}

\begin{theorem}
  \label{Class(0)_Hierarchy}$\forall t \in [\omega, \infty) .G^0 (t) = A^0
  (t)$.
\end{theorem}

\begin{proof}
  We show $\forall t \in [\omega, \infty) .G^0 (t) = A^0 (t)$ by induction on
  $([\omega, \infty), <)$.

  Let $t \in [\omega, \infty)$ be arbitrary and consider $\alpha \in
  \tmop{Class} (0)$ and $n \in [1, \omega)$ such that $t \in [\alpha n, \alpha
  n + \alpha)$.

  Suppose $\forall s \in t \cap [\omega, \infty) .G^0 (s) = A^0 (s)$. \ \ \ \
  \ \ \ {\tmstrong{(IH)}}

  {\tmstrong{Case}} $\tmmathbf{t = \alpha}$.
  
  Then $G^0 (\alpha) = \{\beta \in \tmop{Class} (0) | - \alpha + \alpha <
  \beta \leqslant \alpha \wedge \beta \leqslant_1 g (0, \alpha, \beta)
  (\alpha) + 1\} =$
  
  \ \ \ \ \ \ \ \ \ \ \ \ \ \ \ $= \{\beta \in \tmop{Class} (0) | \alpha
  \geqslant \beta \leqslant_1 \beta + 1\}$ $\underset{\text{proposition }
  \ref{characterization_of_alpha<less>_1alpha+1}}{=}$ $(\tmop{Lim}
  \tmop{Class} (0)) \cap (\alpha + 1) = A^0 (\alpha)$.

  {\tmstrong{Case}} $\tmmathbf{t = l + 1}$ {\tmstrong{for some}} $\tmmathbf{l
  \in [\alpha n, \alpha n + \alpha)}$.
  
  Then $G^0 (l + 1) = \{\beta \in \tmop{Class} (0) | - \alpha n + (l + 1) <
  \beta \leqslant \alpha \wedge \beta \leqslant_1 g (0, \alpha, \beta) (l + 1)
  + 1\}$ $\underset{\tmop{corollary} \ref{cor_<less>^0_equivalences}}{=}$\\
  $\{\beta \in \tmop{Class} (0) | - \alpha n + (l + 1) < \beta \leqslant
  \alpha \wedge$\\
  \ \ \ \ \ $\beta \in \tmop{Lim} \{\gamma \in \tmop{Class} (0) | - \beta n +
  g (0, \alpha, \beta) (l + 1) < \gamma \wedge \gamma
  \leqslant_1 g (0, \beta, \gamma) (g (0, \alpha, \beta) (l
  + 1))\} \}=$\\
  $\{\beta \in \tmop{Class} (0) | - \alpha n + (l + 1) < \beta \leqslant
  \alpha \wedge$\\
  \ \ \ \ \ $\beta \in \tmop{Lim} \{\gamma \in \tmop{Class} (0) | - \beta n +
  (\beta n + ( - \alpha n + l + 1)) < \gamma \wedge$\\
  \ \ \ \ \ \ \ \ \ \ \ \ \ \ \ \ \ \ \ \ \ $\gamma \leqslant_1 \gamma n + (-
  \beta n +  (\beta n + ( - \alpha n + l + 1)))\}\}=$\\
  $\{\beta \in \tmop{Class} (0) | - \alpha n + (l + 1) < \beta \leqslant
  \alpha \wedge$\\
  \ \ \ \ \ $\beta \in \tmop{Lim} \{\gamma \in \tmop{Class} (0) | - \alpha n
  + (l + 1) < \gamma \wedge \gamma \leqslant_1 \gamma n + ( - \alpha n + l +
  1)\}\}=$\\
  $\tmop{Lim} \{\gamma \in \tmop{Class} (0) | - \alpha n + (l + 1) < \gamma
  \leqslant \alpha \wedge \gamma \leqslant_1 \gamma n + ( - \alpha n + l +
  1)\}=$\\
  $\tmop{Lim} \{\gamma \in \tmop{Class} (0) | - \alpha n + (l + 1) < \gamma
  \leqslant \alpha \wedge \gamma \leqslant_1 g (0, \alpha, \gamma) (l + 1)\}
  =$\\
  $\tmop{Lim} \{\gamma \in \tmop{Class} (0) | - \alpha n + l < \gamma
  \leqslant \alpha \wedge \gamma \leqslant_1 g (0, \alpha, \gamma) (l) + 1\}
  =$\\
  Lim$G^0 (l)$ $\underset{\text{by (IH)}}{=}$ Lim$A^0 (l) = A^0 (l + 1)$.

  {\tmstrong{Case}} $\tmmathbf{\alpha < t \in [\alpha n, \alpha n + \alpha)
  \cap \tmop{Lim}}$.
  
  In order to show $G^0 (t) = A^0 (t)$, we make some preparations first. Note
  \\
  $G^0 (t) = \{\beta \in \tmop{Class} (0) | - \alpha n + t < \beta \leqslant
  \alpha \wedge \beta \leqslant_1 g (0, \alpha, \beta) (t) + 1\} =$, as in the
  previous case,
  
  \ \ $= \tmop{Lim} \{\gamma \in \tmop{Class} (0) | - \alpha n + l < \gamma
  \leqslant \alpha \wedge \gamma \leqslant_1 g (0, \alpha, \gamma) (t)\}$. \ \
  \ \ \ \ \ {\tmstrong{(*0)}}

  On the other hand, let's show \\
  $\forall \xi \in \tmop{Class} (0) . - \alpha n + t < \xi \leqslant \alpha
  \wedge \xi \leqslant_1 g (0, \alpha, \gamma) (t) \}\Longrightarrow$\\
  \hspace*{15ex} $ \xi \in
  \bigcap_{s \in \{q \in (\alpha, t) | T (0, \alpha, q) \cap \alpha \subset
  \xi\}} A^0 (s)$ \ \ \ \ \ \ \ {\tmstrong{(*1)}}

  Let $\xi \in \tmop{Class} (0)$ be such that $- \alpha n + t < \xi \leqslant
  \alpha \wedge \xi \leqslant_1 g (0, \alpha, \gamma) (t) \}$. \ \ \ \ \ \ \
  {\tmstrong{(*2)}} \\
  Let $s \in \{q \in (\alpha, t) | T (0, \alpha, q) \cap \alpha \subset \xi\}$
  be arbitrary and let $m \in [1, n]$ be such that \\
  $s \in [\alpha m, \alpha m + \alpha)$. Then clearly $- \alpha m + s < \xi
  \leqslant \alpha$ and \\
  $\xi \leqslant \xi m + (- \alpha m + s + 1) \leqslant \xi n + (- \alpha n +
  t) = g (0, \alpha, \gamma) (t)$; the latter implies, by (*2) and
  $\leqslant_1$-connectedness, $\xi \leqslant_1 \xi m + (- \alpha m + s + 1) =
  (\xi m + ( - \alpha m + s)) + 1 = g (0, \alpha, \gamma) (s) + 1$. This shows
  $\xi \in \{\gamma \in \tmop{Class} (0) | - \alpha m + s < \gamma \leqslant
  \alpha \wedge \gamma \leqslant_1 g (0, \alpha, \gamma) (s) + 1\} = G^0
  (s)$ $\underset{\text{by our (IH)}}{=}$ $A^0 (s)$ and since this was done for
  arbitrary $s \in \{q \in (\alpha, t) | T (0, \alpha, q) \cap \alpha \subset
  \xi\}$, it follows \\
  $\xi \in \bigcap_{s \in \{q \in (\alpha, t) | T (0, \alpha, q) \cap \alpha
  \subset \xi\}} A^0 (s)$. Hence (*1) holds.

  Now we show $\{\gamma \in \tmop{Class} (0) | - \alpha n + l < \gamma
  \leqslant \alpha \wedge \gamma \leqslant_1 g (0, \alpha, \gamma) (t)\} =$ 
  $\{r \in \tmop{Class} (0) \cap (\alpha + 1) | - \alpha n + t < r \in
  \bigcap_{s \in \{q \in (\alpha, t) | T (0, \alpha, q) \cap \alpha \subset
  r\}} A^0 (s)\}$ \ \ \ \ \ \ \ {\tmstrong{(*3)}}

  Note from (*1) follows immediately that the contention $'' \subset''$ of
  (*3) holds. Let's see that the contention $'' \supset''$ also holds:

  Let $\beta \in \{r \in \tmop{Class} (0) \cap (\alpha + 1) | - \alpha n + t
  < r \in \bigcap_{s \in \{q \in (\alpha, t) | T (0, \alpha, q) \cap \alpha
  \subset r\}} A^0 (s)\}$ be arbitrary. Then $\beta \in \tmop{Class} (0)
  \wedge - \alpha n + l < \beta \leqslant \alpha$ \ \ \ \ \ \ \
  {\tmstrong{(*4)}} and \\
  $\beta \in \bigcap_{s \in \{q \in (\alpha, t) | T (0, \alpha, q) \cap \alpha
  \subset \beta\}} A^0 (s)$ $\underset{\text{by (IH)}}{=}$ $\bigcap_{s \in \{q \in
  (\alpha, t) | T (0, \alpha, q) \cap \alpha \subset \beta\}} G^0 (s) =$\\
  $\bigcap_{s \in \{q \in (\alpha, t) | T (0, \alpha, q) \cap \alpha \subset
  \beta\}} \{\gamma \in \tmop{Class} (0) | T (0, \alpha, s) \cap \alpha
  \subset \gamma \leqslant \alpha \wedge \gamma \leqslant_1 g (0, \alpha,
  \gamma) (s)\}$ \ \ \ \ \ \ \ {\tmstrong{(*5)}}.\\
  This way, for the sequences $(\delta_s)_{s \in I}$ and $(\xi_s)_{s \in I}$
  defined as\\
  $I : = \left\{ \begin{array}{l}
    (0, - \alpha n + t) \text{ \tmop{iff} } t > \alpha n\\
    (0, \beta) \text{ \tmop{iff} } t = \alpha n
  \end{array} \right.$, \\
  $\delta_s \assign \left\{ \begin{array}{l}
    \alpha n + s \text{ \tmop{iff} } t > \alpha n\\
    \alpha (n - 1) + s \text{ \tmop{iff} } t = \alpha n
  \end{array} \right.$ \\ and \\  
  $\xi_s \assign \left\{ \begin{array}{l}
    \beta n + s \text{\tmop{iff}} t > \alpha n\\
    \beta (n - 1) + s \text{\tmop{iff}} t = \alpha n
  \end{array} \right.$, \\
  we have that, by (*4) and (*5), \\
  $\forall s \in I.T (0, \alpha, \delta_s) \cap \alpha \subset \beta
  \leqslant_1 g (0, \alpha, \beta) (\delta_s) = \xi_s$ and \\
  $\xi_s$ $\underset{cof}{\longhookrightarrow}$ $\left. \left\{ \begin{array}{l}
    \beta n + (- \alpha n + t) \text{ \tmop{iff} } t > \alpha n\\
    \beta n \text{ \tmop{iff} } t = \alpha n
  \end{array} \right. \right\} = g (0, \alpha, \beta) (t)$. From all this and
  using $\leqslant_1$-continuity, we conclude $\alpha \geqslant \beta \in
  \tmop{Class} (0) \wedge - \alpha n + t < \beta \leqslant_1 = g (0, \alpha,
  \beta) (t)$, that is,
  $\beta \in \{\gamma \in \tmop{Class} (0) | - \alpha n + t < \gamma \leqslant
  \alpha \wedge \gamma \leqslant_1 g (0, \alpha, \gamma) (t)\} = G^0 (t)$.
  Since this was done for arbitrary $\beta \in \{r \in \tmop{Class} (0) \cap
  (\alpha + 1) | - \alpha n + t < r \in \bigcap_{s \in \{q \in (\alpha, t) | T
  (0, \alpha, q) \cap \alpha \subset r\}} A^0 (s)\}$, then $'' \supset''$ of
  (*3) also holds.

  Finally, it is now very easy to see that $G^0 (t) = A^0 (t)$ holds:\\
  $G^0 (t)$ $\underset{\text{by (*0)}}{=}$ $\tmop{Lim} \{\gamma \in \tmop{Class} (0) |
  - \alpha n + t < \gamma \leqslant \alpha \wedge \gamma \leqslant_1 g (0,
  \alpha, \gamma) (t)\}$ $\underset{\text{by (*3)}}{=}$
  
  \ \ \ \ $= \tmop{Lim} \{r \in \tmop{Class} (0) \cap (\alpha + 1) | - \alpha
  n + t < r \in \bigcap_{s \in \{q \in (\alpha, t) | T (0, \alpha, q) \cap
  \alpha \subset r\}} A^0 (s)\}$\\
  \hspace*{5ex} $= A^0 (t)$.
\end{proof}

\begin{proposition}
  Let $\kappa$ be a regular non-countable ordinal. \\
  Then $\forall t \in [\kappa, \kappa \omega) .A^0 (t)$ is closed unbounded in $\kappa$.
\end{proposition}

\begin{proof}
  By induction on $([\kappa, \kappa \omega), <)$. One needs to work a little
  bit with the usual properties of closed unbounded classes.
\end{proof}

As a final result here, we show that there are ordinals $\alpha \in
\tmop{Class} (0)$ such that $\alpha <_1 \alpha \omega$.

\begin{proposition}
  Let $\kappa$ be a regular non-countable ordinal and $\alpha \assign \min
  \tmop{Class} (0) = \omega$. Then
  \begin{enumeratenumeric}
    \item $\bigcap_{t \in [\kappa, \kappa \omega) \wedge T (0, \kappa, t) \cap
    \kappa \subset \alpha} A^0 (t) = \{\gamma \in \tmop{Class} (0) \cap
    (\kappa + 1) | \gamma <_1 \gamma \omega\}$.
    
    \item $\{\gamma \in \tmop{Class} (0) | \gamma <_1 \gamma \omega\}$ is
    closed unbounded in $\kappa$.
  \end{enumeratenumeric}
\end{proposition}

\begin{proof}
  Let $\kappa$ and $\alpha$ be as stated.
  
  {\noindent}1.\\  
  To show $\bigcap_{t \in [\kappa, \kappa \omega) \wedge T (0, \kappa, t) \cap
  \kappa \subset \alpha} A^0 (t) \subset \{\gamma \in \tmop{Class} (0) \cap
  (\kappa + 1) | \gamma \leqslant_1 \gamma \omega\}$. \ \ \ \ \ {\tmstrong{(*0)}}

  Let $\beta \in \bigcap_{t \in [\kappa, \kappa \omega) \wedge T (0, \kappa,
  t) \cap \kappa \subset \alpha} A^0 (t) = \bigcap_{t \in [\kappa, \kappa
  \omega) \wedge T (0, \kappa, t) \cap \kappa \subset \alpha} G^0 (t) =$\\
  $\bigcap_{t \in [\kappa, \kappa \omega) \wedge T (0, \kappa, t) \cap \kappa
  \subset \alpha} \{\gamma \in \tmop{Class} (0) | T (0, \kappa, t) \cap \kappa
  \subset \gamma \leqslant \kappa \wedge \gamma \leqslant_1 g (0, \kappa,
  \gamma) (t) + 1\}$. Notice from this follows that $\forall n \in [1, \omega)
  .T (0, \kappa, \kappa n) \cap \kappa \subset \alpha \leqslant \beta
  \leqslant \kappa \wedge \beta \leqslant_1 g (0, \kappa, \beta) (\kappa n) +
  1 = \beta n + 1$; therefore, since the sequence $(\beta n + 1)_{n \in [1,
  \omega)}$ is cofinal in $\beta \omega$, we get, by $\leqslant_1$-continuity,
  $\kappa \geqslant \beta \leqslant_1 \beta \omega$. Since this was done for
  arbitrary $\beta \in \bigcap_{t \in [\kappa, \kappa \omega) \wedge T (0,
  \kappa, t) \cap \kappa \subset \alpha} A^0 (t)$, then (*0) follows.

  To show\\
  $\bigcap_{t \in [\kappa, \kappa \omega) \wedge T (0, \kappa, t) \cap
  \kappa \subset \alpha} A^0 (t) \supset \{\gamma \in \tmop{Class} (0) \cap
  (\kappa + 1) | \gamma \leqslant_1 \gamma \omega\}$. \ \ \ \ \ \ \ {\tmstrong{(*1)}}

  Let $\beta \in \{\gamma \in \tmop{Class} (0) \cap (\kappa + 1) | \gamma
  \leqslant_1 \gamma \omega\}$. \ \ \ \ \ \ \ {\tmstrong{(*2)}}
  
  Let $t \in [\kappa, \kappa \omega) \wedge T (0, \kappa, t) \cap \kappa
  \subset \alpha$ be arbitrary and let $n \in [1, \omega)$ be such that $t \in
  [\kappa n, \kappa n + \kappa)$. Then $T (0, \kappa, t) \cap \kappa =\{-
  \kappa n + t\} \subset \alpha \leqslant \beta \leqslant \beta n + (- \kappa
  n + t) + 1 < \beta (n + 1) < \beta \omega$ and then, by (*2) and
  $\leqslant_1$-connectedness, we get $T (0, \kappa, t) \cap \kappa \subset
  \beta \leqslant \kappa \wedge \beta \leqslant_1 \beta n + (- \kappa n + t) +
  1 = g (0, \kappa, \beta) (t) + 1$, that is, $\beta \in G^0 (t) = A^0 (t)$.
  Since this was done for arbitrary $\beta \in \{\gamma \in \tmop{Class} (0)
  \cap (\kappa + 1) | \gamma \leqslant_1 \gamma \omega\}$ and for arbitrary $t
  \in [\kappa, \kappa \omega) \wedge T (0, \kappa, t) \cap \kappa \subset
  \alpha$, then we have shown that (*1) holds.

  Hence, by (*0) and (*1) the theorem holds.

  {\noindent}2.\\
  Left to the reader.
\end{proof}

{\newpage}

\appendix\section*{Appendix}

The main goal of this appendix is to clarify how our definition of 
$\alpha <_1 \beta$ based on the notion of isomorfisms 
is equivalent to the assertion that $(\alpha, <, +, <_1)$ is a
$\Sigma_1$-substructure of $(\beta, <, +, <_1)$. For this, it will be
important the kind of language where one works. In the end, we will achieve our goal 
by showing theorem \ref{A<less>_s_1B_equivalence} which,
given certain language $\mathcal{L}$ and corresponding structures $\mathcal{A}$ and
$\mathcal{B}$ for it, characterizes when $\mathcal{A}$ is a $\Sigma_1$
substructure of $\mathcal{B}$. So let us first introduce all the notions that we need.

\section{The language $\mathcal{L}$}

\subsection{Sintax}

In what follows, let us denote as $\bar{R}$, $\bar{f}$ and $\bar{c}$ to a
finite set (or list) of relational, functional and constant symbols,
respectively. As usual, we call the triad $\left\langle \bar{R}, \bar{f},
\bar{c} \right\rangle$ signature.

The terms of $\mathcal{L}$ are build up based on an numerable set of
individual variables $\left\{ w_1, w_2, \ldots \right\}$ and on the individual
constant symbols $\bar{c}$ as follows

\begin{definition}
  (Atomic terms and terms). The atomic terms and terms of our language
  $\mathcal{L}$ with signature $\left\langle \bar{R}, \bar{f}, \bar{c}
  \right\rangle$ are defined as:
  
  {\tmstrong{Atomic terms}}
  
  - Every variable $w$ in $\left\{ w_1, w_2, \ldots \right\}$ is an atomic
  term.
  
  - Every constant $c$ in $\bar{c}$ is an atomic term.
  
  {\tmstrong{Terms}}
  
  - Every atomic term is a term
  
  - If $f$ is a functional symbol of arity $n$ and $s_1, \ldots, s_n$ are
  atomic terms, then $f \left( s_1, \ldots, s_n \right)$ is a term.
\end{definition}

On the other hand, the formulas of $\mathcal{L}$ are given in the following
way

\begin{definition}
  (Atomic formulas and formulas).
  
  {\tmstrong{Atomic formulas}}
  
  - If $R$ is a relational constant symbol of arity $n$ and and $s_1, \ldots,
  s_n$ are terms, then $R \left( s_1, \ldots, s_n \right)$ is an atomic
  formula.
  
  - If $t_1$ is a term and $t_2$ is a constant or a variable, then $t_1
  \approx t_2$ and $t_2 \approx t_1$ are atomic formulas.
  
  {\tmstrong{Formulas}}
  
  - Every atomic formula is a formula
  
  - Given the formulas $F_1$ and $F_2$ and the variable $w$, the following are
  formulas: $F_1 \vee F_2$, $F_1 \wedge F_2$, $F_1 \rightarrow F_2$, $\neg
  F_1$, $\exists w.F_1$, $\forall w.F_1$
\end{definition}

{\newpage}

\subsection{Semantics}

\begin{definition}
  (Non-closed structures). Let $\mathcal{A}= \langle A, \bar{f}^{\mathcal{A}},
  \bar{R}^{\mathcal{A}}, \bar{c}^{\mathcal{A}} \rangle$ be a structure for our
  language $\mathcal{L}$ with the peculiarity that for a universe $U \supset
  A$, the functions $f^{\mathcal{A}} \in \bar{f}^{\mathcal{A}}$ have domain
  and codomain $U$.

  Our interst in these structures lies in the fact that, for $a \in A$,
  $f^{\mathcal{A}} \left( a \right)$ not necessarily belongs to $A$. We will
  call $\mathcal{A}$ a non-closed structure for the language $\mathcal{L}$.
  (On the other hand, for $R^{\mathcal{A}} \in \bar{R}^{\mathcal{A}}$ of arity
  $n$ and $c^{\mathcal{A}} \in \bar{c}^{\mathcal{A}}$, we require
  $c^{\mathcal{A}} \in A$ and $R^{\mathcal{A}} \subset A^n$).
\end{definition}

As usual, for a list of variables $\bar{w} = \left( w_1, \ldots, w_n \right)$
and a list of values $\bar{l} = \left( l_1, \ldots, l_n \right)$ in $A$, we
denote as $[ \bar{w} \assign \bar{l}] : \{w_1, \ldots, w_n \} \rightarrow A$
to the assignment of the variables $\bar{w}$ to the values $\bar{l}$ in $A$.
Moreover, for a term $t$, we denote as $t [ \bar{w} \assign \bar{l}]$ to the
usual application of the assignment $[ \bar{w} \assign \bar{l}]$ to the term
$t$. {\tmstrong{Note this value might not lie in $A$.}}

With respect to the satisfiability notions, we treat equality in $\mathcal{L}$
in the cannonical way: it has a fixed interpretation, namely, the identity. In
general, we treat the satisfaction of a formula $\mathcal{L}$ by a non-closed
structure exaclty in the same way as it is done with structures.

\section{Isomorphisms}

\begin{definition}
  Let $\mathcal{A}= \langle A, \bar{f}^{\mathcal{A}}, \bar{R}^{\mathcal{A}},
  \bar{c}^{\mathcal{A}} \rangle$ and $\mathcal{B}= \langle B,
  \bar{f}^{\mathcal{B}}, \bar{R}^{\mathcal{B}}, \bar{c}^{\mathcal{B}} \rangle$
  be non-closed structures of a language $\mathcal{L}$. An isomorphism between
  them is a function $h : A \longrightarrow B$ such that:
  
  + $h$ is a bijection.
  
  + $h (c^{\mathcal{A}}) = c^{\mathcal{B}}$ for any individual constant symbol
  $c$.
  
  + For any functional symbol $f$ of arity $n$ and any $a_1, \ldots, a_n \in
  A$
  
  \ \ $\bullet$ $f^{\mathcal{A}} (a_1, \ldots, a_n) \in A \Longleftrightarrow
  f^B (h (a_1), \ldots, h (a_n)) \in B$
  
  \ \ $\bullet$ If $f^{\mathcal{A}} (a_1, \ldots, a_n) \in A$, then $h
  (f^{\mathcal{A}} (a_1, \ldots, a_n)) = f^{\mathcal{B}} (h (a_1), \ldots,
  (a_n))$.
  
  + For any $a_1, \ldots, a_n \in A$, and for any relational symbol $R$ of
  arity $n$,
  
  \ \ $R^{\mathcal{A}} (a_1, \ldots, a_n) \Longleftrightarrow R^{\mathcal{B}}
  (h (a_1), \ldots, h (a_n))$.
\end{definition}

\begin{remark}
  \label{h^-1_is_iso_too}It is easy to see that in case $h : A \longrightarrow
  B$ is an isomorphism between $\mathcal{A}= \langle A, \bar{f}^{\mathcal{A}},
  \bar{R}^{\mathcal{A}}, \bar{c}^{\mathcal{A}} \rangle$ and $\mathcal{B}=
  \langle B, \bar{f}^{\mathcal{B}}, \bar{R}^{\mathcal{B}},
  \bar{c}^{\mathcal{B}} \rangle$, then $h^{- 1} : B \longrightarrow A$ is an
  isomorphism between $\mathcal{B}= \langle B, \bar{f}^{\mathcal{B}},
  \bar{R}^{\mathcal{B}}, \bar{c}^{\mathcal{B}} \rangle$ and $\mathcal{A}=
  \langle A, \bar{f}^{\mathcal{A}}, \bar{R}^{\mathcal{A}},
  \bar{c}^{\mathcal{A}} \rangle$.
\end{remark}

To link assignments of a non-closed structure $\mathcal{A}$ with assignments
of another non-closed structure $\mathcal{B}$ which is isomorphic to the
former, we introduce the following

\begin{definition}
  Let $\mathcal{A}$ be a non-closed structure for a language $\mathcal{L}$.
  Let $t$ be a term of $\mathcal{L}$ whose free variables are $w_1, \ldots,
  w_n$ and let $[ \bar{w} \assign \bar{l}] : \{w_1, \ldots, w_n \} \rightarrow
  A$ be an assignment of the free variables of $t$ in $A$. We say that
  $\tmmathbf{[ \bar{w} \assign \bar{l}]}$ {\tmstrong{evaluates}}
  $\tmmathbf{t}$ {\tmstrong{in}} $\tmmathbf{A}$ whenever $s [ \bar{w} \assign
  \bar{l}] \in A$ for any subterm $s$ of $t$ (observe this means also that $t
  [ \bar{w} \assign \bar{l}] \in A$). When $t [ \bar{w} \assign \bar{l}] \nin
  A$, but $s [ \bar{w} \assign \bar{l}] \in A$ for any other subterm $s$ of
  $t$, we say that $\tmmathbf{[ \bar{w} \assign \bar{l}]}$ {\tmstrong{quasi
  evaluates}} $\tmmathbf{t}$ {\tmstrong{in}} $\tmmathbf{A}$.
\end{definition}

\begin{remark}
  \label{remark_all_quasi_evaluate}Note that for any term $t$ of our language
  $\mathcal{L}$ and any non-closed structure $\mathcal{A}= \langle A,
  \bar{f}^{\mathcal{A}}, \bar{R}^{\mathcal{A}}, \bar{c}^{\mathcal{A}}
  \rangle$, any assignment $[ \bar{w} \assign \bar{l}] : \{w_1, \ldots, w_n \}
  \rightarrow A$ of the free variables of $t$ in $A$ is an assignment that
  quasi evaluates $t$ in $A$.
\end{remark}

Given the previous notion, we can show

\begin{lemma}
  \label{Isom.Terminos}Let $\mathcal{A}= \langle A, \bar{f}^{\mathcal{A}},
  \bar{R}^{\mathcal{A}}, \bar{c}^{\mathcal{A}} \rangle$ and $\mathcal{B}=
  \langle B, \bar{f}^{\mathcal{B}}, \bar{R}^{\mathcal{B}},
  \bar{c}^{\mathcal{B}} \rangle$ be non-closed structures of the language
  $\mathcal{L}$ such that there is an isomorphism $h : A \rightarrow B$. Let
  $t$ be a term of $\mathcal{L}$ whose free variables are $w_1, \ldots, w_n$
  and let $[ \bar{w} \assign \bar{l}] : \{w_1, \ldots, w_n \} \rightarrow A$
  be an assignment that evaluates $t$ in $A$.\\
  Then $[ \bar{w} \assign \overline{h (l)}]$ evaluates $t$ in $\mathcal{B}$
  and $h (t [ \bar{w} \assign \bar{l}]) = t [ \bar{w} \assign \overline{h
  (l)}]$.
\end{lemma}

\begin{proof}
  By induction over the terms of $\mathcal{L}$.\\
  + For $t$ a variable or a constant it is direct.\\
  + Suppose $t$ has the form $f (t_1, \ldots, t_n)$.
  
  Note $f (t_1, \ldots, t_n) [ \bar{w} \assign \bar{l}] = f^{\mathcal{A}}
  (t_1 [ \bar{w} \assign \bar{l}], \ldots, t_n [ \bar{w} \assign \bar{l}]) \in
  A$ (because $[ \bar{w} \assign \bar{l}]$ evaluates $f (t_1, \ldots, t_n)$)
  and since $h$ is an isomorphism, then \\
  $h (f (t_1, \ldots, t_n) [ \bar{w} \assign \bar{l}]) =$\\
  $h (f^{\mathcal{A}} (t_1 [ \bar{w} \assign \bar{l}], \ldots, t_n [ \bar{w}
  \assign \bar{l}])) =$\\
  $f^{\mathcal{B}} (h (t_1 [ \bar{w} \assign \bar{l}]), \ldots, h (t_n [
  \bar{w} \assign \bar{l}])) \in B$. \ \ \ \ \ \ \ (*1)

  But by induction hypothesis $[ \bar{w} \assign \overline{h (l)}]$ evaluates
  $t_i$ in $\mathcal{B}$ for all $i \in \{1, \ldots, n\}$ and \\
  $h (t_i [ \bar{w} \assign \bar{l}]) = t_i [ \bar{w} \assign \overline{h
  (l)}]$; therefore from (*1) we have that \\
  $f^{\mathcal{B}} (t_1 [ \bar{w} \assign \overline{h (l)}]), \ldots, t_n [
  \bar{w} \assign \overline{h (l)}]) = f^{\mathcal{B}} (h (t_1 [ \bar{w}
  \assign \bar{l}]), \ldots, h (t_n [ \bar{w} \assign \bar{l}])) \in B$. \ \ \
  \ \ \ (*2)

  (*2) shows that $[ \bar{w} \assign \overline{h (l)}]$ evaluates $f (t_1,
  \ldots, t_n)$ in $B$. Moreover, from (*1) and (*2) we get $h (f (t_1,
  \ldots, t_n) [ \bar{w} \assign \bar{l}]) = f^{\mathcal{B}} (t_1 [ \bar{w}
  \assign \overline{h (l)}]), \ldots, t_n [ \bar{w} \assign \overline{h (l)}])
  \underset{\text{\tmop{by} \tmop{definition}}}{=} f (t_1, \ldots, t_n) [
  \bar{w} \assign \overline{h (l)}]$.
\end{proof}

{\newpage}

\subsection{Isomorphisms and satisfiability}

\begin{theorem}
  \label{A|=F[w:=l]_iff_B|=F[w:=hl]}Let $\mathcal{A}= \langle A,
  \bar{f}^{\mathcal{A}}, \bar{R}^{\mathcal{A}}, \bar{c}^{\mathcal{A}} \rangle$
  and $\mathcal{B}= \langle B, \bar{f}^{\mathcal{B}}, \bar{R}^{\mathcal{B}},
  \bar{c}^{\mathcal{B}} \rangle$ be non-closed structures of $\mathcal{L}$
  such that there is an isomorphism $h : A \rightarrow B$. Let $F$ be a
  formula without quantifiers of $\mathcal{L}$ and suppose $w_1, \ldots, w_n$
  are all the free variables in $F$. Let $[ \bar{w} \assign \bar{l}] : \{w_1,
  \ldots, w_n \} \rightarrow A$ be an assignment of the free variables of $F$
  in $A$. Then $\mathcal{A} \vDash F [ \bar{w} \assign \bar{l}]
  \Longleftrightarrow \mathcal{B} \vDash F [ \bar{w} \assign \overline{h
  (l)}]$
\end{theorem}

\begin{proof}
  Let $[ \bar{w} \assign \bar{l}]$ be an assignment satisfying the assumptions
  of the theorem. We proceed by induction on the formula $F$.

  + F is an atomic formula $R (t_1, \ldots, t_n)$.

  To show $\mathcal{A} \vDash F [ \bar{w} \assign \bar{l}] \Longrightarrow
  \mathcal{B} \vDash F [ \bar{w} \assign \overline{h (l)}]$ \ \ \ \ \ \ \ (*0)
  
  $\mathcal{A} \vDash R (t_1, \ldots, t_n) [ \bar{w} \assign \bar{l}]
  \Longleftrightarrow R^{\mathcal{A}} (t_1 [ \bar{w} \assign \bar{l}], \ldots,
  t_n [ \bar{w} \assign \bar{l}])$. \ \ \ \ \ \ \ (*0.1)
  
  But note (*0.1) means, in particular, that $t_1 [ \bar{w} \assign \bar{l}],
  \ldots, t_n [ \bar{w} \assign \bar{l}] \in A$; then, since by remark
  \ref{remark_all_quasi_evaluate} $[ \bar{w} \assign \bar{l}]$ quasi evaluates
  all the terms appearing in $F$, it follows that $[ \bar{w} \assign \bar{l}]$
  evaluates $t_1, \ldots, t_n$ in $\mathcal{A}$. From this, (*0.1) and the
  fact that \ $h$ is an isomorphism we obtain
  
  $R^{\mathcal{B}} (h (t_1 [ \bar{w} \assign \bar{l}]), \ldots, h (t_n [
  \bar{w} \assign \bar{l}])) \Longleftrightarrow$ (by lemma
  \ref{Isom.Terminos})
  
  $R^{\mathcal{B}} (t_1 [ \bar{w} \assign \overline{h (l)}], \ldots, t_n [
  \bar{w} \assign \overline{h (l)}]) \Longleftrightarrow \mathcal{B} \vDash F
  [ \bar{w} \assign \overline{h (l)}]$.
  
  This shows (*0).

  To show $\mathcal{A} \vDash F [ \bar{w} \assign \bar{l}] \Longleftarrow
  \mathcal{B} \vDash F [ \bar{w} \assign \overline{h (l)}]$ \ \ \ \ \ \ \
  (*0.2)
  
  $\mathcal{B} \vDash R (t_1, \ldots, t_n) [ \bar{w} \assign \overline{h
  \left( l \right)}] \Longleftrightarrow R^{\mathcal{B}} (t_1 [ \bar{w}
  \assign \overline{h \left( l \right)}], \ldots, t_n [ \bar{w} \assign
  \overline{h \left( l \right)}])$. \ \ \ \ \ \ \ (*0.3)
  
  Now, since by remark \ref{remark_all_quasi_evaluate} an arbitrary proper
  subterm $s$ of the terms $t_1, \ldots, t_n$ is evaluated by $[ \bar{w}
  \assign \bar{l}]$ in $\mathcal{A}$, then by lemma \ref{Isom.Terminos} we get
  that $s$ is evaluated by $[ \bar{w} \assign \overline{h \left( l \right)}]$
  in $\mathcal{B}$. By this and (*0.3) we conclude, just as in the previous
  case, that $[ \bar{w} \assign \overline{h \left( l \right)}]$ evaluates
  $t_1, \ldots, t_n$ in $\mathcal{B}$. This, (*0.3) and the fact that $h^{-
  1}$ is an isomorphism (by remark \ref{h^-1_is_iso_too}), imply that
  
  $R^{\mathcal{A}} (h^{- 1} (t_1 [ \bar{w} \assign \overline{h \left( l
  \right)}]), \ldots, h^{- 1} (t_n [ \bar{w} \assign \overline{h \left( l
  \right)}])) \Longleftrightarrow$ (by lemma \ref{Isom.Terminos})
  
  $R^{\mathcal{A}} (t_1 [ \bar{w} \assign \overline{h^{- 1} \left( h (l)
  \right)}], \ldots, t_n [ \bar{w} \assign \overline{h^{- 1} \left( h (l)
  \right)}]) \Longleftrightarrow$
  
  $R^{\mathcal{A}} (t_1 [ \bar{w} \assign \bar{l}], \ldots, t_n [ \bar{w}
  \assign \bar{l}]) \Longleftrightarrow \mathcal{A} \vDash F [ \bar{w} \assign
  \bar{l}]$.
  
  This shows (*0.2).

  + $F$ is an atomic formula $t_1 \approx t_2$ with both $t_1$ and $t_2$ being
  either an individual constant or a variable. Then it follows very easily
  that $\mathcal{A} \vDash (t_1 \approx t_2) [ \bar{w} \assign \bar{l}]
  \Longleftrightarrow \mathcal{B} \vDash (t_1 \approx t_2) [ \bar{w} \assign
  \overline{h \left( l \right)}]$.

  + F is an atomic formula $t_1 \approx t_2$ with $t_2$ a constant or a
  variable and $t_1$ of the form $f \left( s_1, \ldots, s_q \right)$. Note
  that by remark \ref{remark_all_quasi_evaluate}, $[ \bar{w} \assign \bar{l}]$
  evaluates $s_1, s_2, \ldots, s_q, t_2$ in $\mathcal{A}$. \ \ \ \ (*1)\\
  Moreover, by (*1) and lemma \ref{Isom.Terminos}, $[ \bar{w} \assign
  \overline{h (l)}]$ evaluates $s_1,, \ldots, s_q, t_2$ in $\mathcal{B}$. \ \
  \ \ \ \ \ (*1.2)

  To show $\mathcal{A} \vDash (t_1 \approx t_2) [ \bar{w} \assign \bar{l}]
  \Longrightarrow \mathcal{B} \vDash \left( t_1 \approx t_2 \right) [ \bar{w}
  \assign \overline{h (l)}]$. \ \ \ \ \ \ \ (*2)
  
  $\mathcal{A} \vDash (t_1 \approx t_2) [ \bar{w} \assign \bar{l}]
  \Longleftrightarrow t_1 [ \bar{w} \assign \bar{l}] = t_2 [ \bar{w} \assign
  \bar{l}] \underset{\text{by (*1)}}{\in} A \Longrightarrow$
  
  $h (t_1 [ \bar{w} \assign \bar{l}]) = h (t_2 [ \bar{w} \assign \bar{l}])
  \in B \Longrightarrow$ (by (*1), previous line and lemma
  \ref{Isom.Terminos})
  
  $t_1 [ \bar{w} \assign \overline{h (l)}] = t_2 [ \bar{w} \assign \overline{h
  (l)}] \Longrightarrow \mathcal{B} \vDash \left( t_1 \approx t_2 \right) [
  \bar{w} \assign \overline{h (l)}]$.
  
  This shows (*2).

  To show $\mathcal{A} \vDash (t_1 \approx t_2) [ \bar{w} \assign \bar{l}]
  \Longleftarrow \mathcal{B} \vDash \left( t_1 \approx t_2 \right) [ \bar{w}
  \assign \overline{h (l)}]$. \ \ \ \ \ \ \ (*3)
  
  $\mathcal{B} \vDash \left( t_1 \approx t_2 \right) [ \bar{w} \assign
  \overline{h (l)}] \Longrightarrow t_1 [ \bar{w} \assign \overline{h (l)}] =
  t_2 [ \bar{w} \assign \overline{h (l)}] \underset{\text{by (*1.2)}}{\in} B
  \Longrightarrow$
  
  $h^{- 1} \left( t_1 [ \bar{w} \assign \overline{h (l)}] \right) = h^{- 1}
  \left( t_2 [ \bar{w} \assign \overline{h (l)}] \right) \in A
  \Longrightarrow$
  
  (by (*1.2), previous line, remark \ref{h^-1_is_iso_too} and lemma
  \ref{Isom.Terminos} used with the isomorphism $h^{- 1}$ and the assignment
  $[ \bar{w} \assign \overline{h (l)}]$)
  
  $t_1 [ \bar{w} \assign \bar{l}] = t_1 [ \bar{w} \assign \overline{h^{- 1}
  \left( h (l) \right)}] = t_2 [ \bar{w} \assign \overline{h^{- 1} \left( h
  (l) \right)}] = t_2 [ \bar{w} \assign \bar{l}] \Longrightarrow$
  
  $\mathcal{A} \vDash (t_1 \approx t_2) [ \bar{w} \assign \bar{l}]$.
  
  This shows (*3).

  + $F$ is an atomic formula $t_2 \approx t_1$ with $t_2$ a constant or a
  variable and $t_1$ of the form $f \left( s_1, \ldots, s_q \right)$. The the
  proof is just as the previous case.

  + The case for the logical connectives follows immediatly by the induction
  hypothesis.
\end{proof}

\section{Substructures and $\Sigma_1$ substructures}

\begin{definition}
  Let $\mathcal{A}= \langle A, \bar{f}^{\mathcal{A}}, \bar{R}^{\mathcal{A}},
  \bar{c}^{\mathcal{A}} \rangle$ and $\mathcal{B}= \langle \mathcal{B},
  \bar{f}^{\mathcal{B}}, \bar{R}^{\mathcal{B}}, \bar{c}^{\mathcal{B}} \rangle$
  be non-closed structures for $\mathcal{L}$. $\mathcal{A}$ is (non-closed)
  substructure of $\mathcal{B}$ iff $A \subset B$, $\bar{f}^{\mathcal{A}} =
  \bar{f}^{\mathcal{B}} |_A$, $\bar{R}^{\mathcal{A}} \subset
  \bar{R}^{\mathcal{B}}$, and $\bar{c}^{\mathcal{A}} = \bar{c}^{\mathcal{B}}$.
\end{definition}

\begin{definition}
  Let $F$ be a formula of $\mathcal{L}$ such that $s_1, \ldots, s_n$ are terms
  appearing in $F$. Let $t_1, \ldots, t_n$ be terms. Then we denote by $F
  \left\langle s_1 \assign t_1, \ldots, s_n \assign t_n \right\rangle$ to the
  formula obtained by the syntactical substitution, for any $i \in \left\{ 1,
  \ldots, n \right\}$, of all the ocurrences of the term $s_i$ by $t_i$ in
  $F$.
\end{definition}

\begin{definition}
  Let \ $\mathcal{A}= \langle A, \bar{f}^{\mathcal{A}}, \bar{R}^{\mathcal{A}},
  \bar{c}^{\mathcal{A}} \rangle$ be a non-closed structure for $\mathcal{L}$
  and $F$ a formula of $\mathcal{L}$. Suppose the terms $t_1, \ldots, t_n$
  appear in $F$. For $l_1, \ldots, l_n \in A$ we define the formula $F
  \left\langle t_1 \assign l_1, \ldots, t_n \assign l_n \right\rangle$ which
  results by the syntactical substitution, for any $i \in \left\{ 1, \ldots, n
  \right\}$, of all the ocurrences of the term $t_i$ in $F$ by the constant
  $l_i$.

  Note that, formally, $F \left\langle t_1 \assign l_1, \ldots, t_n \assign
  l_n \right\rangle$ does not belong to the language $\mathcal{L}$. The idea
  is very simple: {\tmstrong{We}} just {\tmstrong{convey}} that, any time that
  we have a formula like $F$, we consider $F \left\langle t_1 \assign l_1,
  \ldots, t_n \assign l_n \right\rangle$ as a ``formula of $\mathcal{L}$ with
  parameters $l_1, \ldots, l_n$''. Any of the parameters $l_i$ is simply an
  element of $A$ that behaves as a term with a fixed value under any
  assignment $\left[ \bar{w} \assign \bar{e} \right]$, namely, $l_i \left[
  \bar{w} \assign \bar{e} \right] = l_i$.
\end{definition}

\begin{definition}
  Similarly as in the previous definitions, consider the non-closed structures
  $\mathcal{A}= \langle A, \bar{f}^{\mathcal{A}}, \bar{R}^{\mathcal{A}},
  \bar{c}^{\mathcal{A}} \rangle$, $\mathcal{B}= \langle B, \bar{f}^B,
  \bar{R}^{\mathcal{B}}, \bar{c}^{\mathcal{B}} \rangle$ and a formula $F$ of
  $\mathcal{L}$ with parameters $l_1, \ldots, l_n \in A$. Moreover, for any $i
  \in \left\{ 1, \ldots, n \right\}$, let $d_i$ be either a term of
  $\mathcal{L}$ or a parameter $d_i \in B$. Then we denote as $F \left\langle
  l_1 \assign d_1, \ldots, l_n \assign d_n \right\rangle$ to the formula with
  parameters resulting by the syntactical substitution, for any $i \in \left\{
  1, \ldots, n \right\}$, of all the ocurrences of the parameter $l_i$ by
  $d_i$ in $F$.
\end{definition}

The reason of the previous (somewhat annoying) definitions is because we need
them to ennunciate the main notion we want to characterize:

\begin{definition}
  Let \ $\mathcal{A}= \langle A, \bar{f}^{\mathcal{A}}, \bar{R}^{\mathcal{A}},
  \bar{c}^{\mathcal{A}} \rangle$ and $\mathcal{B}= \langle B, \bar{f}^B,
  \bar{R}^{\mathcal{B}}, \bar{c}^{\mathcal{B}} \rangle$ be non-closed
  structures for our language $\mathcal{L}$. $\mathcal{A}$ is sigma one
  non-closed substructure of $\mathcal{B}$, which we abbreviate as usual
  $\mathcal{A} \prec_{\Sigma_1} \mathcal{B}$, if and only if the following two
  statments hold:
  \begin{enumeratenumeric}
    \item $\mathcal{A}$ is (non-closed) substructure of $\mathcal{B}$;
    
    \item For any quantifier free formula $F$ with $n + m$ different free
    variables \\
    $x_1, \ldots, x_n, y_1, \ldots, y_m$ and for any $l_1, \ldots, l_m \in A$,
    \\
    $\mathcal{B} \vDash \exists x_1, \ldots x_n .F \left\langle y_1 \assign
    l_1, \ldots, y_m \assign l_m \right\rangle\Longleftrightarrow $ \\ 
    $\mathcal{A}
    \vDash \exists x_1, \ldots x_n .F \left\langle y_1 \assign l_1, \ldots,
    y_m \assign l_m \right\rangle$.
  \end{enumeratenumeric}
  Note that 2. simply states that $\mathcal{B}$ is model of a $\Sigma_1$
  sentence of $\mathcal{L}$ with parameters in $\mathcal{A}$ if and only if
  $\mathcal{A}$ is model of the same sentence.
\end{definition}

We can finally present the theorem that is our main interest:

\begin{theorem}
  \label{A<less>_s_1B_equivalence}Let \ $\mathcal{A}= \langle A,
  \bar{f}^{\mathcal{A}}, \bar{R}^{\mathcal{A}}, \bar{c}^{\mathcal{A}} \rangle$
  and $\mathcal{B}= \langle B, \bar{f}^B, \bar{R}^{\mathcal{B}},
  \bar{c}^{\mathcal{B}} \rangle$ be non-closed structures for our language
  $\mathcal{L}$. Then:
  
  $\mathcal{A}$ is a $\Sigma_1$ (non-closed) substructure of $\mathcal{B}$
  (that is, $\mathcal{A} \prec_{\Sigma_1} \mathcal{B}$)
  
  $\Longleftrightarrow$
  
  $\mathcal{A}$ is (non-closed) substructure of $\mathcal{B}$ and whenever $X$
  is a finite subset of $A$ and $Y$ is a finite subset of $B \backslash A$,
  there exists a subset $\hat{Y} $ of $A$ and an isomorphism\\
  $h : X \cup Y \rightarrow X \cup \hat{Y}$ from the non-closed structure\\
  $\langle X \cup Y, \bar{f}^{\mathcal{B}} |_{X \cup Y}, \bar{R}^{\mathcal{B}}
  |_{X \cup Y}, \bar{c}^B |_{X \cup Y} \rangle$ to the non-closed structure\\
  $\langle X \cup \hat{Y}, \bar{f}^{\mathcal{B}} |_{X \cup \hat{Y}},
  \bar{R}^{\mathcal{B}} |_{X \cup \hat{Y}}, \bar{c}^B |_{X \cup \hat{Y}}
  \rangle$ such that $h (x) = x$ for any $x \in X$.
\end{theorem}

\begin{proof}
  Let $\mathcal{L}$, $\mathcal{A}$ and $\mathcal{B}$ be as stated. Since
  $\mathcal{L}$ is of finite signature, then there exists a natural number $M
  \in \mathbbm{N}$ such that the arities of all the relational and functional
  symbols in $\mathcal{L}$ is less or equal to $M$. This way, for any natural
  numbers $n, m \in \left[ 1, M \right]$, let Rel$_n$ be the set of relational
  symbols of $\mathcal{L}$ of arity $n$ and let Func$_m$ be the set of
  functional symbols of $\mathcal{L}$ of arity $n$.

  Now we show the direction $\Longrightarrow$) of the theorem. \ \ \ \
  {\tmstrong{(*)}}

  Suppose $\mathcal{A} \prec_{\Sigma_1} \mathcal{B}$. \ \ \ \ \ \ \
  {\tmstrong{(*0)}}

  So $\mathcal{A}$ is a substructure of $\mathcal{B}$ and we only have to
  prove the isomorphisms-related issue. Let $X \subset_{\tmop{fin}} A$ and \
  $Y \subset_{\tmop{fin}} B \backslash A$ be arbitrary. Moreover, suppose $X =
  \left\{ b_1, \ldots, b_l \right\}$ and $Y = \left\{ c_1, \ldots, c_k
  \right\}$ for some $l, k \in \mathbbm{N}$.

  Consider the formula $\Gamma$ with parameters in $\mathcal{B}$ defined as
  
  $\underset{n \in \left[ 1, M \right] \wedge R \in \tmop{Rel}_n \wedge
  R^{\beta} \left( q \right) \tmop{for} q \in \left( X \cup Y
  \right)^n}{\bigwedge} R \left( q \right) \wedge$
  
  $\underset{n \in \left[ 1, M \right] \wedge R \in \tmop{Rel}_n \wedge \neg
  R^{\beta} \left( q \right) \tmop{for} q \in \left( X \cup Y
  \right)^n}{\bigwedge} \neg R \left( q \right) \wedge$

  $\underset{n \in \left[ 1, M \right] \wedge f \in \tmop{Func}_n \wedge
  f^{\beta} \left( q \right) = t \tmop{for} q \in \left( X \cup Y \right)^n, t
  \in X \cup Y}{\bigwedge} f (q) \approx t \wedge$
  
  $\underset{n \in \left[ 1, M \right] \wedge f \in \tmop{Func}_n \wedge
  f^{\beta} \left( q \right) \neq t \tmop{for} q \in \left( X \cup Y
  \right)^n, t \in X \cup Y}{\bigwedge} \neg \left( f (q) \approx t \right)$.

  Let $z_1, \ldots, z_k$ be $k$ different variables of $\mathcal{L}$. Note
  $\mathcal{B} \vDash \Gamma$ and therefore $\mathcal{B} \vDash \exists z_1,
  \ldots z_k . \Gamma \left\langle c_1 \assign z_1, \ldots, c_k \assign z_k
  \right\rangle$. But $\Gamma \left\langle c_1 \assign z_1, \ldots, c_k
  \assign z_k \right\rangle$ is a $\Sigma_1$ sentence of $\mathcal{L}$ with
  parameters $b_1, \ldots, b_l$ in $\mathcal{A}$, and therefore, by our
  hypotheis (*0), $\mathcal{A} \vDash \exists z_1, \ldots z_k . \Gamma
  \left\langle c_1 \assign z_1, \ldots, c_k \assign z_k \right\rangle$, which
  means there is $\left( a_1, \ldots, a_k \right) \in A^k$ such that
  $\mathcal{A} \vDash \Gamma \left\langle c_1 \assign z_1, \ldots, c_k \assign
  z_k \right\rangle \left[ z_1 \assign a_1, \ldots, z_k \assign a_k \right]$.
  \ \ \ \ \ \ \ {\tmstrong{(*1)}}

  Let it be $\hat{Y} \assign \left\{ a_1, \ldots, a_k \right\}$.
  
  We now show that the function $h : X \cup Y \xrightarrow[\begin{array}{l}
    X \ni x \longmapsto x\\
    Y \ni c_i \longmapsto a_i
  \end{array}]{} X \cup \hat{Y}$ is the function we are looking for.
  
  $h$ is bijective. $h$ is injective because for $i, j \in \left[ 1, k
  \right]$, $i \neq j$, the formula $\neg \left( z_i \approx z_j \right)$ is a
  subformula of $\Gamma$. Moreover, from the definition of $h$ it is clearly
  surjective.\\

  Clearly $h \left( x \right) = x$ for any $x \in X$.

  Now, let $R \in \tmop{Rel}_n$, $e_1, \ldots, e_n \in X \cup Y$ and $n \in
  \left[ 1, M \right]$.

  To show $R^{\mathcal{B}} |_{X \cup Y} \left( e_1, \ldots, e_n \right)
  \Longleftrightarrow R^{\mathcal{B}} |_{X \cup \hat{Y}} \left( h \left( e_1
  \right), \ldots, h \left( e_n \right) \right)$. \ \ \ \ \ \ \
  {\tmstrong{(*2)}}

  $\left. \Longrightarrow \right)$
  
  $R^{\mathcal{B}} |_{X \cup Y} \left( e_1, \ldots, e_n \right)
  \Longrightarrow R^{\mathcal{B}} \left( e_1, \ldots, e_n \right)
  \Longrightarrow$ $R \left( e_1, \ldots, e_n \right)$ is \\ 
  subformula of $\Gamma$ $\Longrightarrow$  $R \left( e_1, \ldots, e_n \right) \left\langle c_1 \assign z_1, \ldots, c_k
  \assign z_k \right\rangle$ is subformula \\ of $\Gamma \left\langle c_1 \assign
  z_1, \ldots, c_k \assign z_k \right\rangle$ $\xRightarrow[\text{by (*1)}]{}$ \\
  $\mathcal{A} \vDash R \left( e_1, \ldots, e_n \right) \left\langle c_1
  \assign z_1, \ldots, c_k \assign z_k \right\rangle \left[ z_1 \assign a_1,
  \ldots, z_k \assign a_k \right]$, but note that the latter is exactly the
  same as $\mathcal{A} \vDash R \left( h \left( e_1 \right), \ldots, h \left(
  e_n \right) \right)$ and so\\
  $R^{\mathcal{B}} |_{X \cup \hat{Y}} \left( h
  \left( e_1 \right), \ldots, h \left( e_n \right) \right)$.

  $\left. \Longleftarrow \right)$
  
  We show $R^{\mathcal{B}} |_{X \cup \hat{Y}} \left( h \left( e_1 \right),
  \ldots, h \left( e_n \right) \right) \Longrightarrow R^{\mathcal{B}} |_{X
  \cup Y} \left( e_1, \ldots, e_n \right)$ by contrapositive. Suppose
  $R^{\mathcal{B}} |_{X \cup Y} \left( e_1, \ldots, e_n \right)$ doesn't hold.
  Then $\neg R \left( e_1, \ldots, e_n \right)$ is subformula of $\Gamma$ and
  then $\neg R \left( e_1, \ldots, e_n \right) \left\langle c_1 \assign z_1,
  \ldots c_k \assign z_k \right\rangle$ is subformula of $\Gamma \left\langle
  c_1 \assign z_1, \ldots, c_k \assign z_k \right\rangle$. Thus, by (*1),\\
  $\mathcal{A} \vDash \neg R \left( e_1, \ldots, e_n \right) \left\langle c_1
  \assign z_1, \ldots, c_k \assign z_k \right\rangle \left[ z_1 \assign a_1,
  \ldots, z_k \assign a_k \right]$ which is exactly the same as $\mathcal{A}
  \vDash \neg R \left( h \left( e_1 \right), \ldots, h \left( e_n \right)
  \right)$; this way, \\
  $R^{\mathcal{B}} |_{X \cup \hat{Y}} \left( h \left( e_1
  \right), \ldots, h \left( e_n \right) \right)$ doesn't hold.

  All the previous shows that (*2) holds.

  Now, let $f \in \tmop{Func}_n$, $e_1, \ldots, e_n \in X \cup Y$ and $n \in
  \left[ 1, M \right]$.

  Let's suppose $f^{\mathcal{B}} |_{X \cup Y} \left( e_1, \ldots, e_n \right)
  \in X \cup Y.$ \ \ \ \ \ \ \ {\tmstrong{(*3)}}
  
  To show $f^{\mathcal{B}} |_{X \cup \hat{Y}} \left( h \left( e_1 \right),
  \ldots, h \left( e_n \right) \right) \in X \cup \hat{Y}$ and
  
  $f^{\mathcal{B}} |_{X \cup \hat{Y}} \left( h \left( e_1 \right), \ldots, h
  \left( e_n \right) \right) = h \left( f^{\mathcal{B}} |_{X \cup Y} \left(
  e_1, \ldots, e_n \right) \right)$. \ \ \ \ \ \ \ {\tmstrong{(*3.1)}}

  By (*3), $f^{\mathcal{B}} |_{X \cup Y} \left( e_1, \ldots, e_n \right) =
  f^{\mathcal{B}} \left( e_1, \ldots, e_n \right) = d$ for some $d \in X \cup
  Y$. Then the formula $f \left( e_1, \ldots, e_n \right) \approx d$ is a
  subformula of $\Gamma$ and therefore $\left( f \left( e_1, \ldots, e_n
  \right) \approx d \right) \left\langle c_1 \assign z_1, \ldots c_k \assign
  z_k \right\rangle$ is subformula of\\
  $\Gamma \left\langle c_1 \assign z_1,
  \ldots, c_k \assign z_k \right\rangle$. Then, by (*1),\\
  $\mathcal{A} \vDash
  \left( f \left( e_1, \ldots, e_n \right) \approx d \right) \left\langle c_1
  \assign z_1, \ldots c_k \assign z_k \right\rangle \left[ z_1 \assign a_1,
  \ldots, z_k \assign a_k \right]$, which \\
  means $f^{\mathcal{A}} \left( h
  \left( e_1 \right), \ldots, h \left( e_n \right) \right) = h \left( d
  \right) \in X \cup \hat{Y}$. Note the latter equality is exactly as the one
  in (*3.1), since $\mathcal{A}$ is substructure of $\mathcal{B}$.
  
  This shows the two assertions in (*3.1).

  Let's suppose $f^{\mathcal{B}} |_{X \cup Y} \left( e_1, \ldots, e_n \right)
  \nin X \cup Y$. \ \ \ \ \ \ \ {\tmstrong{(*3.2)}}
  
  To show $f^{\mathcal{B}} |_{X \cup \hat{Y}} \left( h \left( e_1 \right),
  \ldots, h \left( e_n \right) \right) \nin X \cup \hat{Y}$. \ \ \ \ \ \
  {\tmstrong{(*3.3)}}
  
  Let $d \in X \cup \hat{Y}$ be arbitrary. Similarly as before, (*3.2) implies
  that the formula $\neg \left( f \left( e_1, \ldots, e_n \right) \approx h^{-
  1} \left( d \right) \right)$ is a subformula of $\Gamma$. So\\
  $\neg \left( f
  \left( e_1, \ldots, e_n \right) \approx h^{- 1} \left( d \right) \right)
  \left\langle c_1 \assign z_1, \ldots c_k \assign z_k \right\rangle$ is
  subformula of \\
  $\Gamma \left\langle c_1 \assign z_1, \ldots, c_k \assign z_k
  \right\rangle$. Then, by (*1), \\
  $\mathcal{A} \vDash \neg \left( f \left( e_1,
  \ldots, e_n \right) \approx h^{- 1} \left( d \right) \right) \left\langle
  c_1 \assign z_1, \ldots c_k \assign z_k \right\rangle \left[ z_1 \assign
  a_1, \ldots, z_k \assign a_k \right]$, \\
  which means $f^{\mathcal{B}} |_{X
  \cup \hat{Y}} \left( h \left( e_1 \right), \ldots, h \left( e_n \right)
  \right) \neq h \left( h^{- 1} \left( d \right) \right) = d$. Since we have
  done this for arbitrary $d \in X \cup \hat{Y}$, we have shown (*3.3).

  All of the previous shows that $h$ is indeed an isomorphism with $h \left( x
  \right) = x$ for any $x \in X$ and therefore, we have shown (*).

  Now we show the direction $\Longleftarrow$) of the theorem. \ \ \ \ \ \ \
  {\tmstrong{(**)}}
  
  So assume the right hand side of the double implication \\
  asserting the theorem. \ \ \ \ \ \ \ {\tmstrong{(*4)}}

  We want to show that $\mathcal{A} \prec_{\Sigma_1} \mathcal{B}$. By
  hypothesis, $\mathcal{A}$ is a non-closed substructure of $\mathcal{B}$, so
  it is only left to show that:
  
  For any quantifier free formula $F$ with $n + m$ different free variables\\
  $u_1, \ldots, u_n, w_1, \ldots, w_m$ and for any $l_1, \ldots, l_m \in A$,\\
  $\mathcal{B} \vDash \exists u_1 \ldots u_n .F \left\langle w_1 \assign l_1,
  \ldots, w_m \assign l_m \right\rangle \Longleftrightarrow$\\ $ \mathcal{A} \vDash
  \exists u_1 \ldots u_n .F \left\langle w_1 \assign l_1, \ldots, w_m \assign
  l_m \right\rangle$ \ \ \ \ \ \ \ {\tmstrong{(*5)}}

  So let's show (*5).
  
  Let $F$ be an arbitrary quantifier free formula whose free (different to
  each other) variables are $u_1, \ldots, u_n, w_1, \ldots, w_m$ and let it be
  $l_1, \ldots, l_m \in A$. Clearly $\mathcal{A} \vDash \exists u_1 \ldots u_n
  .F \left\langle w_1 \assign l_1, \ldots, w_m \assign l_m \right\rangle$
  implies \\
  $\mathcal{B} \vDash \exists u_1 \ldots u_n .F \left\langle w_1
  \assign l_1, \ldots, w_m \assign l_m \right\rangle$ because $\mathcal{A}$ is
  a substructure of $\mathcal{B}$. So we actually only have to show that\\
  $\mathcal{B} \vDash \exists u_1 \ldots u_n .F \left\langle w_1 \assign l_1,
  \ldots, w_m \assign l_m \right\rangle \Longrightarrow$\\
  $\mathcal{A} \vDash
  \exists u_1 \ldots u_n .F \left\langle w_1 \assign l_1, \ldots, w_m \assign
  l_m \right\rangle$ \ \ \ \ \ \ {\tmstrong{ (*6)}}

  To show (*6).
  
  Suppose $\mathcal{B} \vDash \exists u_1 \ldots u_n .F \left\langle w_1
  \assign l_1, \ldots, w_m \assign l_m \right\rangle$. \ \ \ \ \ \ \
  {\tmstrong{(*7)}}
  
  Then there exist $e_1, \ldots, e_n \in B$ such that\\ 
  $\mathcal{B} \vDash F
  \left\langle w_1 \assign l_1, \ldots, w_m \assign l_m \right\rangle \left[
  u_1 \assign e_1, \ldots, u_n \assign e_n \right]$. \ \ \ \ \ \ \
  {\tmstrong{(*8)}}

  On the other hand, let $t_1, \ldots, t_q$ be all the terms and subterms
  appearing in $F \left\langle w_1 \assign l_1, \ldots, w_m \assign l_m
  \right\rangle$ such that $t_1, \ldots, t_q$ are evaluated in $B$ by $\left[
  u_1 \assign e_1, \ldots, u_n \assign e_n \right]$. Moreover, for any $i \in
  \left\{ 1, \ldots, q \right\}$, let $b_i \in B$ be such that $t_i \left[ u_1
  \assign e_1, \ldots, u_n \assign e_n \right] = b_i$.

  To make more manageable our notation, let's abbreviate\\
  $\left\langle w_1 \assign l_1, \ldots, w_m \assign l_m \right\rangle$ and $\left[ u_1 \assign
  e_1, \ldots, u_n \assign e_n \right]$ as $\left\langle \bar{w} \assign
  \bar{l} \right\rangle$ and $\left[ \bar{u} \assign \bar{e} \right]$,
  respectively.

  Let $Y \assign \left\{ b_i |i \in \left\{ 1, \ldots, q \right\} \wedge b_i
  \in B \backslash A \right\} = \left\{ y_1, \ldots, y_j \right\}$ and \\
  $X \assign \left\{ b_i |i \in \left\{ 1, \ldots, q \right\} \wedge b_i \in A
  \right\} = \left\{ x_1, \ldots, x_k \right\}$. By (*4), there exists a set
  $\hat{Y} \assign \left\{ \hat{y}_1, \ldots, \hat{y}_j \right\} \subset A$
  such that the function $\underset{\begin{array}{l}
    x_i \longmapsto x_i\\
    y_i \longmapsto \hat{y}_i
  \end{array}}{h : X \cup Y \longrightarrow X \cup \hat{Y}}$ is an
  isomorphism. \ \ \ \ \ \ \ {\tmstrong{(*9)}}

  Now observe that (*8) and our definition of $Y$ and $X$ imply that\\
  $X \cup Y \vDash F \left\langle \bar{w} \assign \bar{l} \right\rangle \left[
  \bar{u} \assign \bar{e} \right]$, i.e., $X \cup Y \vDash F \left[ \bar{w}
  \assign \bar{l}, \bar{u} \assign \bar{e} \right]$ and therefore, by (*9) and
  theorem \ref{A|=F[w:=l]_iff_B|=F[w:=hl]}, $X \cup \hat{Y} \vDash F \left[
  \bar{w} \assign \overline{h \left( l \right)}, \bar{u} \assign \overline{h
  \left( e \right)} \right]$; so, using that $\forall a \in Y \cap A.h \left(
  a \right) = a$, we get $X \cup \hat{Y} \vDash F \left\langle \bar{w} \assign
  \bar{l} \right\rangle \left[ \bar{u} \assign \overline{h \left( e \right)}
  \right]$. But then $X \cup \hat{Y} \vDash \exists u_1 \ldots u_n .F
  \left\langle \bar{w} \assign \bar{l} \right\rangle$ and since $X \cup
  \hat{Y} \subset A$, we conclude $\mathcal{A} \vDash \exists u_1 \ldots u_n
  .F \left\langle \bar{w} \assign \bar{l} \right\rangle$.

  All of the previous shows (*6). Therefore (*5) is also proven and
  subsequently (**) has been proven too.

  This concludes the proof of the whole theorem.
\end{proof}

\begin{remark}
  In previous theorem \ref{A<less>_s_1B_equivalence}, it is necessary that
  $\mathcal{L}$ contains an equality symbol to show that $h : X \cup Y
  \rightarrow X \cup \hat{Y}$ is injective. To see this, consider a language
  $L_0$ with has only one binary relation $\sim$. Let \ $\mathcal{A}= \langle
  A, \sim^{\mathcal{A}} \rangle$ and $\mathcal{B}= \langle B,
  \sim^{\mathcal{B}} \rangle$ be given as $A \assign \left\{ 0 \right\},
  \sim^{\mathcal{A}} \assign \left\{ \left( 0, 0 \right) \right\}$, $B \assign
  \left\{ 0, 1 \right\}$ and $\sim^{\mathcal{B}} \assign \left\{ \left( 0, 0
  \right), \left( 1, 1 \right), \left( 0, 1 \right), \left( 1, 0 \right)
  \right\}$. Then one can prove:
  \begin{itemizedot}
    \item $\mathcal{A} \prec_{\Sigma_1} \mathcal{B}$
  \end{itemizedot}
  Moreover, for $X \assign \left\{ 0 \right\} \subset_{\tmop{fin}} A$ and $Y
  \assign \left\{ 1 \right\} \subset_{\tmop{fin}} B \backslash A$ the
  following holds:
  \begin{itemizedot}
    \item There exist no $\hat{Y} \subset_{\tmop{fin}} A$ and an isomorfism $h
    : X \cup Y \rightarrow X \cup \hat{Y}$ with $h \left( x \right) = x$ for
    any $x \in X$.
    
    \item For $\hat{Y} \assign \emptyset \subset_{\tmop{fin}} A$, the function
    $l : X \cup Y \xrightarrow[x \longmapsto 0]{} X \cup \hat{Y}$ is an
    homomorphism with $l \left( x \right) = x$ for any $x \in X$.
  \end{itemizedot}
\end{remark}

\begin{remark}
  A more classical version of theorem \ref{A<less>_s_1B_equivalence}, where
  one does not have to deal with the hassles of considering non-closed
  structures, can be stated as follows:
\end{remark}

\begin{theorem}
  Let \ $\mathcal{A}= \langle A, \bar{R}^{\mathcal{A}}, =^{\mathcal{A}}
  \rangle$ and $\mathcal{B}= \langle B, \bar{R}^{\mathcal{B}}, =^{\mathcal{B}}
  \rangle$ be structures for a language $\mathcal{L}$ with equality symbol and
  with a finite number of relational symbols. Then:
  
  $\mathcal{A}$ is $\Sigma_1$ (non-closed) substructure of $\mathcal{B}$ (that
  is, $\mathcal{A} \prec_{\Sigma_1} \mathcal{B}$)
  
  $\Longleftrightarrow$
  
  $\mathcal{A}$ is substructure of $\mathcal{B}$ and whenever $X$ is a finite
  subset of $A$ and $Y$ is a finite subset of $B \backslash A$, there exists a
  subset $\hat{Y} $ of $A$ and an isomorphism\\
  $h : X \cup Y \rightarrow X \cup \hat{Y}$ from $\left\langle X \cup Y,
  \bar{R}^{\mathcal{B}} |_{X \cup Y} \right\rangle$ to $\left\langle X \cup
  \hat{Y}, \bar{R}^{\mathcal{B}} |_{X \cup \hat{Y}} \right\rangle$ such that
  $h (x) = x$ for any $x \in X$.
\end{theorem}

Finally, let us state a final proposition that is very useful while working
with Carlson's $<_1$-relation.

\begin{proposition}
  \label{iso.restriction}Let $(C, \overline{R^{}}^C, \overline{f^{}}^C,
  \bar{c})$, $(Q, \overline{R^{}}^Q, \bar{f}^Q, \bar{q})$ be structures of a
  language $L$. Suppose \\
  $(B, \overline{R^{}}^B, \overline{f^{}}^B, \bar{b}) \subset (C,
  \overline{R^{}}^C, \overline{f^{}}^C, \bar{c})$, that is, $B \subset C$,
  $R^B = R^C \cap B^n$ for any n-ary relation $R^C$, $f^B = f^C |_B$ for any
  function $f^C$ and any distinguished element $b$ of $B$ is a distinguished
  element of $C$.
  
  Suppose $h : (C, \overline{R^{}}^C, \overline{f^{}}^C, \bar{c})
  \longrightarrow (h [C], \overline{R^{}}^{h [C]}, \overline{f^{}}^{h [C]},
  \overline{h (c)}) \subset (Q, \overline{R^{}}^Q, \bar{f}^Q, \bar{q})$ is an
  isomorphism.
  
  Then $h|_B : (B, \overline{R^{}}^B, \overline{f^{}}^B, \bar{b})
  \longrightarrow (h [B], \overline{R^{}}^{h [B]}, \overline{f^{}}^{h [B]},
  \overline{h (b)})$ is an isomorphism.
\end{proposition}

\begin{proof}
  For any $a_1, \ldots, a_n \in B$ and any relation $R^B$ we have $R^B (a_1,
  \ldots, a_n) \Longleftrightarrow R^C (a_1, \ldots, a_n) \Longleftrightarrow
  R^{h [C]} (h (a_1), \ldots, h (a_n)) \Longleftrightarrow R^{h [B]} (h|_B
  (a_1), \ldots, h|_B (a_n))$.
  
  Clearly $b \in B$ is a distinguished element iff $h (b) = h|_B (b) \in h
  [B]$ is a distinguished element.
  
  Let's see that the operations behave also correctly (of course the problem
  is with the closure of such operations):
  
  Let $a_1, \ldots, a_n \in B$. Suppose $f^C (a_1, \ldots, a_n) = f^B (a_1,
  \ldots, a_n) \in B$. Then \\
  $f^{h [C]} (h (a_1), \ldots, h (a_n)) \in h [C]$ and $f^{h [C]} (h (a_1),
  \ldots, h (a_n)) = h (f^C (a_1, \ldots, a_n)) = h (f^B (a_1, \ldots, a_n))$.
  Clearly $h (a_1), \ldots, h (a_n) \in h [B] \subset h [C]$ and so from the
  previous equalities we have $f^{h [B]} (h|_B (a_1), \ldots, h|_B (a_n)) =
  f^{h [C]} (h (a_1), \ldots, h (a_n)) = h (f^B (a_1, \ldots, a_n)) \in h
  [B]$.
  
  Now suppose $f^{h [B]} (h|_B (a_1), \ldots, h|_B (a_n)) \in h [B]$. Then
  there exists $a \in B \subset C$ such that $h (a) = f^{h [B]} (h|_B (a_1),
  \ldots, h|_B (a_n))$. \ \ \ \ \ \ {\tmstrong{(A)}}
  
  On the other hand, $f^{h [C]} (h (a_1), \ldots, h (a_n)) = f^{h [B]} (h|_B
  (a_1), \ldots, h|_B (a_n)) \in h [B] \subset h [C]$; then $f^C (a_1, \ldots,
  a_n)) \in C$ and \\
  $h (f^C (a_1, \ldots, a_n)) = f^{h [C]} (h (a_1), \ldots, h
  (a_n)) = f^{h [B]} (h|_B (a_1), \ldots, h|_B (a_n))$. From this and (A) we
  have found that $h (f^C (a_1, \ldots, a_n)) = h (a)$ and therefore, since
  $h$ is bijective, $f^C (a_1, \ldots, a_n) = a \in B$.
\end{proof}

\end{document}